\documentclass[11pt,twoside]{article}
\usepackage{amssymb}
\usepackage{latexsym}
\usepackage{graphicx,color}
\usepackage{amsmath}
\usepackage{hyperref,amsthm}
\usepackage{epsfig}
\usepackage{float}
\numberwithin{equation}{section} 
\normalsize \makeatletter
\newtheorem{thm}{Theorem}[section]
\newtheorem{lemma}[thm]{Lemma}
\newtheorem{pro}[thm]{Proposition}

\newtheorem{rem}[thm]{Remark}

 \newcommand{\be}{\begin{equation}}
\newcommand{\ee}{\end{equation}}
\newcommand{\bea}{\begin{eqnarray*}}
\newcommand{\eea}{\end{eqnarray*}}
\newcommand{\mR}{\mathbb{R}}

\newcommand{\mC}{\mathbb{C}}
\newcommand{\mE}{\mathbb{E}}
\newcommand{\mP}{\mathbb{P}}
\newcommand{\y}{\widetilde{Y}}
\newcommand{\p}{\partial}
\newcommand*{\ackname}{Acknowledgement}
\makeatletter

\newcommand{\Rmnum}[1]{\expandafter\@slowromancap\romannumeral #1@}
\makeatother

\def\bs{\boldsymbol}
\def\mbf{\mathbf}
\def\tr{\mbox{tr}}
\newenvironment{keyword}{\smallskip\noindent{\bf Keywords.}
                          \hskip\labelsep}{}
\newenvironment{classification}{\smallskip\noindent{\textbf{AMS subject classification:}}
                                 \hskip\labelsep}{}

\title{\textbf {Limiting spectral distribution of renormalized separable
sample covariance matrices when $p/n\to 0$}}

 \author{Lili Wang$^\dag$\footnote{correspondence author: Department of Mathematics,
Zhejiang University, Hangzhou, 310027, China. E-mail: sta.lesleywang@gmail.com; Phone: (+86)13857138955; $\ddag$ Department of Statistics,
University of California, Davis,
Davis, CA 95616. E-mail: debpaul@ucdavis.edu.} ~and Debashis Paul$^\ddag$\\
{\small $\dag$ Zhejiang University and University of California, Davis;} \\
{\small $\ddag$ University of California, Davis}}
\begin{document}
\date{}
\maketitle

\begin{abstract}
We are concerned with the behavior of the eigenvalues of renormalized sample
covariance matrices of the form
$$
C_n=\sqrt{\frac{n}{p}}\left(\frac{1}{n}A_{p}^{1/2}X_{n}B_{n}X_{n}^{*}A_{p}^{1/2}-\frac{1}{n}\tr(B_{n})A_{p}\right)
$$
as $p,n\to \infty$ and $p/n\to 0$, where $X_{n}$ is a $p\times n$ matrix with
i.i.d. real or complex valued entries $X_{ij}$ satisfying $E(X_{ij})=0$,
$E|X_{ij}|^2=1$ and having finite fourth moment. $A_{p}^{1/2}$ is a square-root
of the nonnegative definite Hermitian matrix $A_{p}$, and $B_{n}$ is an
$n\times n$ nonnegative definite Hermitian matrix. We show that the empirical
spectral distribution (ESD) of $C_n$ converges a.s. to a nonrandom limiting
distribution under the assumption that the ESD of $A_{p}$ converges to a
distribution $F^A$ that is not degenerate at zero, and that the first and
second spectral moments of $B_n$ converge. The probability density function of the LSD of $C_{n}$ is derived and it is
shown that it depends on the LSD of $A_{p}$ and the limiting value of
$n^{-1}\tr(B_{n}^2)$. We propose a computational algorithm for evaluating this
limiting density when the LSD of $A_{p}$ is a mixture of point masses. In
addition, when the entries of $X_{n}$ are sub-Gaussian, we derive the limiting
empirical distribution of $\{\sqrt{n/p}(\lambda_j(S_n) - n^{-1}\tr(B_n)
\lambda_j(A_{p}))\}_{j=1}^p$ where $S_n := n^{-1}
A_{p}^{1/2}X_{n}B_{n}X_{n}^{*}A_{p}^{1/2}$ is the sample covariance matrix and
$\lambda_j$ denotes the $j$-th largest eigenvalue, when $F^A$ is a finite
mixture of point masses. These results are utilized to propose a test for the
covariance structure of the data where the null hypothesis is that the joint
covariance matrix is of the form $A_{p} \otimes B_n$ for $\otimes$ denoting the Kronecker product, as well as $A_{p}$ and the first two spectral moments of
$B_n$ are specified. The performance of this test is illustrated through a simulation study.
\end{abstract}

\newpage

\section{Introduction}\label{sec:intro}

In this paper, we obtain the limiting spectral distribution (LSD) and a system
of equations describing the corresponding Stieltjes transforms of renormalized
sample covariance matrices of the form
\begin{equation}\label{eq:C_n}
C_n=\sqrt{\frac{n}{p}}\left(\frac{1}{n}A_{p}^{1/2}X_{n}B_{n}X_{n}^{*}A_{p}^{1/2}-\frac{1}{n}\tr(B_{n})A_{p}\right)
\end{equation}
when $p,n\to \infty$ and $p/n \to 0$, where $X_{n}$ has i.i.d. real or complex
entries with zero mean, unit variance and uniformly bounded fourth moment.
Throughout this paper, for any matrix $M$, we use $M^{*}$ to  denote the
complex conjugate transpose of $M$ if $M$ is complex-valued and transpose of
$M$ if $M$ is real-valued. When $p/n\to c\in (0,+\infty)$ as $n\to \infty$, the
spectral properties of the separable sample covariance matrices, $S_{n} :=
n^{-1} A_{p}^{1/2}X_{n}B_{n}X_{n}^{*}A_{p}^{1/2}$  have been widely
investigated under different assumptions on entries (e.g., Zhang \cite{ZPhd},
Paul and Silverstein \cite{PS}, EL Karoui \cite{Karoui}). The name
``separable'' refers to the fact that the covariance matrix of the vectorized
data matrix $Y_n = A_{p}^{1/2} X_n B_n^{1/2}$ has the separable covariance
$A_{p} \otimes B_n$, where $\otimes$ denotes the Kronecker product between
matrices. Under those circumstances, it is known that the spectral norm of
$S_{n}-\mE S_{n}$ does not converge to zero. However, if $p/n\to 0$, $\parallel
S_{n}-\mE S_{n}\parallel \stackrel{a.s.}{\longrightarrow} 0$. When
$A_{p}=I_{p}, B_{n}=I_{n}$ and $p,n\to \infty$ such that $ p/n\to 0$, the
behavior of empirical spectral distribution (ESD) of $\sqrt{n/p}\left(S_{n}-\mE
S_{n}\right) =\sqrt{n/p}(n^{-1} X_{n}X_{n}^{*}-I_p)$  is similar to
that of a $p\times p$ Wigner matrix $W_{p}$, which has been verified by Bai and
Yin \cite{BY}. Moreover, when $S_{n} =
n^{-1}A_{p}^{1/2}X_{n}X_{n}^{*}A_{p}^{1/2}$, for i.i.d. real entries and under
a finite fourth moment condition, Pan and Gao \cite{PanG} and Bao \cite{Bao}
derived the LSD of
$\sqrt{n/p}(n^{-1}A_{p}^{1/2}X_{n}X_{n}^{*}A_{p}^{1/2}-A_{p})$,
which coincides with that of a generalized Wigner matrix
$p^{-1/2}A_{p}^{1/2}W_{p}A_{p}^{1/2}$ studied by Bai and Zhang \cite{BZ}. Our
work here extends the former result to a more general setting, namely, when
$B_{n}$ is an arbitrary $n\times n$ positive semi-definite matrix whose first
two spectral moments converge to finite positive values as $n\to \infty$, and
the entries of $X_{n}$ are either real or complex. The strategy of the proof of
this result is divided into three parts. We first assume that the entries of
$X_n$ are i.i.d. Gaussian and use a construction analogous to that in Pan and
Gao \cite{PanG} to obtain the form of the approximate deterministic equations
describing the expected Stieltjes transforms, then use a result on
concentration of smooth functions of independent random elements to show that
the Stieltjes transform concentrates around its mean in the general setting
(without the restriction of Gaussianity), and finally utilize the Lindeberg
principle 
to show that the expected Stieltjes transforms in the Gaussian and in the
general case are asymptotically the same. In the process, we also prove the
existence and uniqueness of the system of equations describing the Stieltjes
transform for an arbitrary $F^A$, non-degenerate at zero. Further, we state a
result characterizing the LSD, including the existence and shape of its density
function, by following the approach in Bai and Zhang \cite{BZ}. We also study
the question of fluctuation of the eigenvalues of the sample covariance matrix
$S_n := n^{-1} A_{p}^{1/2}X_{n}B_{n}X_{n}^{*}A_{p}^{1/2}$ itself when the ESD of $A_{p},$ say
$F^{A_{p}}$ and its limit $F^A$ are finite mixtures of point masses.
Specifically, we show that the empirical distribution of
$\{\sqrt{n/p}(\lambda_j(S_n) - n^{-1}\tr(B_n) \lambda_j(A_{p}))\}_{j=1}^p$,
where $\lambda_j$ denotes the $j$-th largest eigenvalue, converges a.s. to a
mixture of rescaled semi-circle laws with mixture weights being the same as the
weights corresponding to the point masses of $F^A$ and the scaling factor
depending on the limiting value of $n^{-1}\tr(B_{n}^2)$ and the atoms of
$F^{A}$.

It should be noted that the data model of the form
$Y_{n}=A_{p}^{1/2}X_{n}B_{n}^{1/2}$, where $X_{n}$ has i.i.d. entries with zero
mean and unit variance,
relates very closely to the \textit{separable covariance model} widely
used in spatio-temporal data modeling, especially for modeling
environmental data (e.g., Kyriakidis and Journel \cite{Ky99},
Mitchell and Gumpertz \cite{Mit03}, Fuentes \cite{Fuentes}, Li et al. \cite{Li08}). The separable covariance model
refers to the fact that for any $p$ sampling locations in space, and any
$n$ observation times, the covariance of the corresponding data
matrix can be expressed in the form $\Sigma_{n,p} = A_{p} \otimes
B_{n}$. In that context, the rows of $Y_{n}$ correspond to spatial
locations while the columns represent the observation times. If furthermore,
the process is Gaussian, which is often assumed in the literature,
then the data matrix $Y_n$ is exactly of the form $A_{p}^{1/2}X_{n}B_{n}^{1/2}$
where $X_{ij}$'s are i.i.d. $N(0,1)$. Assuming a separable covariance structure,
that the process is stationary in space, the sampling locations
cover the entire spatial region under consideration fairly evenly, and
the temporal variation has only short term dependence (not necessarily
stationary), the covariance of the observed data can be expressed in the form
$A_{p}\otimes B_{n}$ where $A_{p}$ and $B_{n}$ satisfy conditions
3, 4 and 5 of our main result in this paper (Theorem \ref{thm:main_LSD}). Moreover,
if the sampling locations are on a spatial grid, then the matrix of eigenvectors of $A_p$
is approximately the Fourier rotation matrix on $\mathbb{R}^p$ and the eigenvalues
are approximately the Fourier transform of the spatial autocovariance kernel
evaluated at certain discrete frequencies related to the grid spacings.

There is a body of literature on the statistical inference for a separable
covariance model, in particular about the tests for separability of the joint
covariance of the data. Notable examples include Dutilleul \cite{Dutilleul}, Lu
and Zimmerman \cite{LZ}, Mitchell et al. (\cite{Mit05}, \cite{Mit06}), Fuentes
\cite{Fuentes},  Roy and Khatree \cite{RK}, Simpson \cite{Simpson} and Li et al. \cite{Li08}.
These tests typically assume joint Gaussianity of the data and often the derivation
of the test statistic requires additional structural assumptions, e.g.,
stationarity of the spatial and temporal processes (Fuentes \cite{Fuentes}).
In addition, the estimation techniques often involve matrix inversions (Dutilleul
\cite{Dutilleul}, Mitchell et al. \cite{Mit06}) which become challenging if the
dimensionality (either $p$ or $n$) is large. We study the problem of tests
involving the separable covariance structure under the framework $p, n\to
\infty$ and $p/n\to 0$. Under this setting, $\parallel n^{-1}Y_{n}Y_{n}^{*}-
n^{-1}\tr(B_{n})A_{p}\parallel \stackrel{a.s.}{\longrightarrow} 0$ and hence we
can infer about the spectral properties of $A_{p}$ from that of the sample
covariance matrix $n^{-1}Y_{n}Y_{n}^{*}$. In particular, we propose to
use the results derived here to construct 
test statistic for testing whether the space-time data follows a specific
separable covariance model, where the null hypothesis is in terms of
specification of $A_{p}$ and the first two spectral moments of $B_n$. Let
$A_{0}$, $\tr(B_{0})$ and $\tr(B_{0}^2)$ be the specified values of $A_{p}$,
$\tr(B_{n})$ and $\tr(B_{n}^2)$ under the null hypothesis, Then this statistic
measures the difference of the ESD of the matrix $\sqrt{n/p}(n^{-1} Y_n Y_n^* -
n^{-1}\tr(B_{0}) A_{0})$, from the LSD of $C_n$ described in (\ref{eq:C_n}),
where the matrix  $X_n$ is assumed to have i.i.d. entries with zero mean and
unit variance, $A_{p} = A_0$, $\tr(B_{n}) = \tr(B_{0})$ and $\tr(B_{n}^2) =
\tr(B_{0}^2)$. We also propose a Monte-Carlo method for determination of the
cut-off values of the test for any given level of significance and analyze the
behavior of the test through simulation. We also carry out a simulation
study with different combinations of $(p,n)$ to empirically assess the rate of
convergence of the ESD under to the LSD as measured by the $L^2$ distance between
these distributions.

\section{Main results}\label{sec:main_results}
Under the framework presented in Section \ref{sec:intro}, our main result in
this paper is about the existence and uniqueness of the LSD of $C_n$ defined
through (\ref{eq:C_n}). The result will be described in terms of the Stieltjes
transform of the matrices. The Stieltjes transform of the empirical spectral
distribution $F^{C_{n}}$ is defined as
$$
s_{n}(z)=\int\frac{1}{x-z}dF^{C_{n}}(x)
$$
for any $z \in \mC^{+}:=\{u+iv:u \in \mathbb{R}, v>0\}$. It will be shown that
the ESD of $C_{n}$ will converge almost surely to a distribution $F$ whose
Stieltjes transform is determined by a system of equations.
\begin{thm}\label{thm:main_LSD}
Suppose that
\begin{enumerate}
\item for every $p$ and $n$, $\{X_{i,j}: 1\leq i \leq p, 1\leq j \leq n\}$ is an
array of i.i.d real or complex valued random variables with
$\mathbb{E}(X_{11})=0$, $\mathbb{E}|X_{11}|^{2}=1$ and $\mathbb{E}|X_{11}|^4<
\infty$;
\item $p=p(n) \to \infty$ with $p(n)/n\to 0$ as $n\to \infty$;
\item $A_{p}$ is a $p\times p$ nonnegative definite Hermitian matrix and $B_{n}$ is a $n\times n$
nonnegative hermitian matrix;
\item
the ESD $F^{A_{p}} \Longrightarrow F^{A}$ as $p \to \infty$ where $F^{A}$ is a
nonrandom distribution function on $\mathbb{R}_+$ that is not degenerate at
zero;
\item
$\parallel B_{n} \parallel$ is bounded above, and $n^{-1} \tr(B_{n}^k)$ for
$k=1,2$  converge to finite positive constants as $n \to \infty$.
\end{enumerate}
Then $F^{C_{n}}$ almost surely converges weakly to a nonrandom distribution $F$
as $n\to \infty$, whose Stieltjes transform $s(z)$ is determined by the
following system of equations:
\begin{equation}\label{equation_system}
\begin{cases}
s(z) = - \int \frac{dF^{A}(a)}{z +\bar{b}_{2} a \beta(z)} \\
\beta(z) =  -\int \frac{ a  dF^{A}(a)}{z + \bar{b}_2  a \beta(z)}
\end{cases}
\end{equation}
for any $z\in \mC^{+}$, where $\bar{b}_2 = \displaystyle\lim_{n \to \infty}
n^{-1}tr(B_n^2)$.
\end{thm}

\begin{rem}\label{rem:b_2}
In (\ref{equation_system}), the constant $\bar{b}_2$ determines the scale of
the support of the LSD $F$. Specifically, the LSD $F$ is related to the LSD
$F_{A,I}$ corresponding to the case $B_{n} = I_n$ (studied by Pan and Gao
\cite{PanG} and Bao \cite{Bao}), by $F(x) = F_{A,I}(\bar{b}_2^{-1/2} x)$ for
all $x \in \mathbb{R}$. Note also that, $F_{A,I}$ coincides with the LSD of the
generalized Wigner matrix $p^{-1/2} A_{p}^{1/2} W_p A_{p}^{1/2}$ analyzed by
Bai and Zhang \cite{BZ}.
\end{rem}

\begin{rem}\label{rem:semicircle}
If $A_{p} = I_p$, the two equations (\ref{equation_system}) reduce to only one,
namely, $s(z)(z + \bar{b}_2 s(z)) = - 1$, which is the equation for a rescaled
semi-circle law $F_{sc}(\cdot;\sqrt{\bar{b}_2})$ with scaling factor
$\sqrt{\bar{b}_2}$, where, for any $\sigma > 0$,
\begin{equation}\label{eq:scaled_semi_circle_law}
F_{sc}(x;\sigma) := F_{sc}(\sigma^{-1} x), \qquad \mbox{for all}~~ x \in
\mathbb{R},
\end{equation}
where $F_{sc}$ denotes the semi-circle law. Notice that, in this case due to
the rotational invariance, the statement of the Theorem \ref{thm:main_LSD}
reduces to the statement that the empirical distribution of
$\{\sqrt{n/p}(\lambda_j(S_n) - n^{-1}\tr(B_n))\}_{j=1}^p$, converges a.s. to
the rescaled semi-circle law with scaling factor $\bar{b}_2$. We present an
interesting generalization of this result in Section
\ref{subsec:fluctuation_eigenvalues}.
\end{rem}

\begin{rem}\label{rem:non_Hermitian_root}
In the statement of Theorem \ref{thm:main_LSD}, the matrix $A_{p}^{1/2}$ needs
not be the Hermitian square root of $A_{p}$. As long as
$(A_{p}^{1/2})(A_{p}^{1/2})^* = A_{p}$, the result will continue to hold. In
particular, $A_{p}^{1/2}$ can be of the form $A_{p}^{1/2} = U \Lambda^{1/2}
V^*$, where $A_{p} = U \Lambda U^*$ is the spectral decomposition of $A_{p}$,
so that $\Lambda$ is a diagonal matrix and $V^* V = U^* U = I_p$. Moreover,
from this it also follows that if $\widetilde V$ is a $p\times q$ matrix with
$q \leq p$ such that $\widetilde V^* \widetilde V = I_q$ where $q \to \infty$
such that $q/p \to \omega \in (0,1]$  and $\widetilde Y_n = \widetilde V^* X_n
B_{n}^{1/2}$ then the ESD of $\sqrt{n/q}(n^{-1} \widetilde Y_n\widetilde Y_n^*
- n^{-1}\tr(B_{n}) I_q)$ converges a.s. to the distribution
$F_{sc}(\cdot;\sqrt{\bar{b}_2})$ introduced in Remark \ref{rem:semicircle}.
\end{rem}

\begin{rem}\label{rem:proof_thm}
Proof of Theorem \ref{thm:main_LSD} can be divided into the following parts:
\begin{enumerate}
\item The spectrum of $A_{p}$ is truncated at a sufficiently large $a_{0}$. This
is done in Section \ref{subsec:truncation_F_A}, following an approach in Bai and
Yin \cite{BY}. It is shown that the ESD of the $C_{n}$ and the matrix
corresponding to the truncated $A_{p}$ are almost surely equivalent.

\item For each $z \in \mC^{+}$, and for $C_n$ corresponding to matrices with
i.i.d. standard Gaussian entries, $\mathbb{E}(s_{n}(z)) \to s(z)$ satisfying
(\ref{equation_system}), which is shown in Section
\ref{subsec:deterministic_part}.

\item
For each $z\in \mC^{+}$, $s_{n}(z) - \mathbb{E}(s_n(z))$ converges almost
surely to zero. This is derived in Section \ref{subsec:random part} through the
use of McDiarmid's inequality (McDiarmid \cite{Mcdiarmid}).

\item
Existence and uniqueness of the solution of the system of equations
(\ref{equation_system}) defining the limiting Stieltjes transform $s(z)$ is
established in \ref{subsec:uniqueness}.

\item
The entries of $X_{n}$ are truncated at $n^{1/4}\epsilon_{p}$  and centered,
where $\epsilon_{p}\to 0$, $p^{1/4}\epsilon_{p}\to \infty$ which does not alter
the LSD. The result is established in the general setting by establishing the
asymptotic negligibility of the difference of $\mathbb{E}(s_n(z))$
corresponding to independent copies of $X_{n}$ with such truncated entries in
Section \ref{subsec:nongaussian}.
\end{enumerate}
\end{rem}


\subsection{Analysis of LSD}\label{subsec:analysis LSD}

The following result characterizes the behavior of the LSD $F$ in Theorem
\ref{thm:main_LSD}.
\begin{pro}\label{prop:LSD_density}
Suppose that $F^{A}(a)\ne I_{[0,\infty)}(a)$ and let $F$ be the LSD of
$F^{C_n}$ as in Theorem \ref{thm:main_LSD}. Then, $F(\{0\})=F^{A}(\{0\})$, and
for any real $x\neq 0$,
\begin{equation*}
s(x)=\lim_{z\in \mC^{+}\to x}s(z), \quad \beta(x)=\lim_{z\in \mC^{+}\to
x}\beta(z)
\end{equation*}
exist such that
\begin{equation*}
s(x)=-\int \frac{dF^{A}(a)}{x+\bar{b}_{2}a\beta(x)},
\end{equation*}
where $\beta(x)$ uniquely solves
\begin{equation}\label{eq:beta_x_equation}
\beta(x)=-\int \frac{adF^{A}(a)}{x+\bar{b}_2a\beta(x)}
\end{equation}
while satisfying $\Im \beta(x)\geq 0$, $x\Re\beta(x)<0$ and $\omega(x) \leq 1$,
where
\begin{equation}\label{eq:omega_x_equation}
\omega(x):=\int\frac{\bar{b}_{2}a^{2}}{|x+\bar{b}_{2}a\beta(x)|^{2}}dF^{A}(a).
\end{equation}
Moreover, we have
\begin{enumerate}
\item $s(x)$ and $\beta(x)$ are continuous on the real line except only at the origin.
\item  $F(x)$ is symmetric and continuously differentiable on
the real line except at the origin and its derivative is given by
\begin{equation}\label{eq:density}
f(x)=-\frac{2\Re\beta(x)\Im\beta(x)}{\pi x}.
\end{equation}
\item The support of $F$, say $\mathcal{S}$ is determined as follows: for any
$x \neq 0$, $x \in \mathcal{S}^c$ (complement of $\mathcal{S}$) if and only if
there exists some $\delta > 0$ such that for all $y \in (x-\delta,x+\delta)$,
$\omega(y) < 1$.
\end{enumerate}
\end{pro}
The proof of this proposition follows along the line of the proof Theorem 1.2
of Bai and Zhang \cite{BZ} with an additional scale factor $\bar{b}_2:=
\displaystyle\lim_{n\to \infty} n^{-1}\tr(B_{n}^2)$. The following lemma, which
is a consequence of property (3) of Proposition \ref{prop:LSD_density},
provides a way of determining the support of the density function.
\begin{lemma}\label{lemma:density_support}
The support of $f(x)$ is the set of $x \in \mathbb{R}$ satisfying $x\Re\beta(x)
< 0$ and $\omega(x)=1$ (equivalently, $\Im\beta(x)>0$).
\end{lemma}
A more direct verification of Lemma \ref{lemma:density_support} is given in
Section \ref{subsec:uniqueness}.

\subsection{Fluctuation of eigenvalues of $S_n$}\label{subsec:fluctuation_eigenvalues}

In certain applications, not only the eigenvalues of $C_n$ but difference of
the eigenvalues of $S_{n}$ from those of $\mathbb{E}(S_{n})$ may be of
interest. Since $\parallel S_n - \mathbb{E}(S_n)\parallel \to 0$, a.s., under
the framework $p/n \to 0$, it is expected that the eigenvalues of $S_{n}$ will
fluctuate around the ``corresponding'' eigenvalues of $\mathbb{E}(S_{n}) =
n^{-1}\tr(B_{n}) A_{p}$. To make this notion more precise, we consider the
setting where there are finitely many distinct eigenvalues of $A_{p}$. Then for
large enough $n$, the eigenvalues of $S_n$ will tend to cluster around these
distinct eigenvalues of $n^{-1}\tr(B_{n}) A_{p}$. Moreover, if both $\parallel
A_{p} \parallel$ and $\parallel B_{n} \parallel$ are bounded,  the proportion
of eigenvalues falling in each cluster will coincide with the proportion of the
corresponding eigenvalue of $A_{p}$ in the ESD of $A_{p}$. This can be seen as
an instance of the spectrum separation phenomenon studied by Bai and
Silverstein \cite{BS} for sample covariance matrices in the setting $p/n \to c
\in (0,\infty)$. Our goal in this subsection is to establish that, if
$\parallel B_n \parallel$ is bounded, and if the probability distribution of
the entries of $X_n$ has sufficiently fast decay in the tails (specifically,
``sub-Gaussian tails''), then the fluctuations of the eigenvalues of $S_{n}$
around the eigenvalues of $n^{-1}\tr(B_{n}) A_{p}$ can be fully characterized,
provided $A_{p}$ has finitely many distinct eigenvalues, the proportion of each
of which converges to a nonzero fraction.

To state the result, we first define a sub-Gaussian random vector (cf.
Vershynin \cite{Vershynin}). A real-valued random vector $\mathbf{y} =
(y_1,\ldots,y_n)^T$ is said to be sub-Gaussian with scale parameter $\sigma
\geq 0$, if for all $\gamma \in \mathbb{R}^n$,
\begin{equation}\label{eq:subgaussian_vector}
\mathbb{E}[\exp(\gamma^T \mathbf{y})] \leq \exp(\parallel \gamma \parallel^2
\sigma^2/2).
\end{equation}
Clearly, if $\mathbf{y}$ has independent coordinates each of which is
sub-Gaussian with scale parameter $\sigma$, then $\mathbf{y}$ is sub-Gaussian.
Moreover, it is easy to see that if $\mathbf{y}$ is sub-Gaussian with scale
parameter $\sigma$, then for any $m \times n$ matrix $D$, the vector $D
\mathbf{y}$ is also sub-Gaussian, with scale parameter $\parallel D \parallel
\sigma$. A complex-valued random vector is sub-Gaussian if and only if both
real and imaginary parts of the vector are sub-Gaussian.

\begin{thm}\label{thm:eigenvalue_fluctuations}
Let $B_{n}$ be a $n\times n$ positive semi-definite matrix such that
$\parallel B_{n} \parallel$ is bounded above, $n^{-1} \tr(B_{n}) \to \bar{b}$
and $n^{-1} \tr(B_{n}^2) \to \bar{b}_2$ as $n\to \infty$. Let $A_{p}$ be a
$p\times p$ positive semidefinite matrix with $m$ distinct eigenvalues
$\alpha_1 > \ldots > \alpha_m \geq 0$ such that $\alpha_1$ is bounded above and
$m$ is fixed, and if $p_j$ denotes the multiplicity of $\lambda_j$, then $p_j/p
\to c_j > 0$ for all $j$, as $p \to \infty$. Let $Y_n = A_{p}^{1/2} X_{n}
B_{n}^{1/2}$ where  $X_{n}$ is a $p\times n$ matrix with i.i.d. real or complex
sub-Gaussian entries $X_{ij}$ satisfying $\mathbb{E}(X_{11}) = 0$ and
$\mathbb{E}|X_{11}|^2 = 1$. In the complex case, we also suppose that the real
and imaginary parts of $X_{ij}$ are independent with variance $1/2$ each. Let
$S_n := n^{-1}Y_nY_n^*$, and let $\lambda_j(D)$ denote the $j$-th largest
eigenvalue of a Hermitian matrix $D$. Then, as $p,n \to \infty$ such that $p/n
\to 0$, the empirical distribution of $\{\sqrt{n/p}(\lambda_j(S_n) -
n^{-1}\tr(B_n) \lambda_j(A_{p}))\}_{j=1}^p$ converges a.s. to a nonrandom
probability distribution $F$ on $\mathbb{R}$ which can be expressed as
$$
F(x) = \sum_{j=1}^m c_j F_{sc}(x; \sqrt{\bar{b}_2 c_j} \alpha_j),
~~~\mbox{for}~x \in \mathbb{R},
$$
where $F_{sc}(\cdot;\sigma)$, for any $\sigma > 0$, is defined in
(\ref{eq:scaled_semi_circle_law}).
\end{thm}

\begin{rem}\label{rem:eigenvalue_fluctuations}
The assumption of sub-Gaussianity of the entries in Theorem
\ref{thm:eigenvalue_fluctuations} is not necessary if $p = o(n^{1/3})$. In that
case, if we only assume the finiteness of $\mathbb{E}|X_{11}|^4$, it can be
shown that the empirical distribution of $\{\sqrt{n/p}(\lambda_j(S_n) -
n^{-1}\tr(B_n) \lambda_j(A_{p}))\}_{j=1}^p$ converges in probability to the
same limit law. This is because, as is seen from the proof given in Section
\ref{sec:proof_eigenvalue_fluctuations}, without loss of generality assuming
$m=2$,  the conclusion follows upon showing that $\sqrt{n/p}\parallel
S_{12}\parallel^2 = o_P(1)$, which can be majorized by $\sqrt{n/p}\parallel
S_{12}\parallel_F^2$, where $\parallel \cdot \parallel_F$ denotes the Frobenius
norm. The latter is $O_P(\sqrt{n/p}(p^2/n)) = o_P(1)$ under the stated
conditions. A stronger conclusion (in the form of a.s. convergence) can be made
with appropriately higher moment conditions.
\end{rem}


\subsection{Application to hypothesis testing}\label{subsec:application}

The results in the previous subsections allow us to develop a test for the
hypothesis that data matrix $Y_n$ has a specific separable covariance
structure. Suppose that the vectorized $Y_n$ has joint covariance $\Sigma_n$.
Then our null hypothesis is that
\begin{equation}\label{eq:null_separable}
H_{0}: \Sigma_n = A_{p} \otimes B_n ~\mbox{with}~ A_{p}=A_{0},
n^{-1}\tr(B_{n})=\bar{b}^0, n^{-1}\tr(B_{n}^2)=\bar{b}_2^0,
\end{equation}
where $A_0$, $\bar{b}^0 > 0$ and $\bar{b}_2^0 \geq (\bar{b}^0)^2$ are
specified. Note that $\bar{b}^0$ and $\bar{b}_2^0$ can be seen as the
(limiting) values of $n^{-1}\tr(B_{0})$ and $n^{-1}\tr(B_{0}^2)$, respectively,
where $A_0 \otimes B_0$ is the covariance of the data $Y_n$ under $H_{0}$.
Thus, $H_{0}$ is a composite hypothesis about $\Sigma_n$. Also note that,
testing for $H_{0}$ is not the same as testing for separability of the data
model since in our setting $A_{p}$ needs to be specified. Later in this
subsection, we discuss potential extensions of the proposed test procedure for
dealing with the null hypothesis of separability, under certain weaker
restrictions on $A_{p}$.

We propose a test statistic that measures the closeness of the empirical
spectrum to the theoretical spectral density under the null hypothesis. If the
data matrix is endowed with the assumed covariance structure in $H_{0}$,
Theorem \ref{thm:main_LSD} 
guarantees the convergence of ESD of $C_n$ to an LSD. In fact, Proposition
\ref{prop:LSD_density}
gives an explicit expression for the aforementioned LSD $F$. Equipped with this
result, in this paper, we propose and study the following test statistics based
on the $L^2$ metric:
\begin{equation}\label{eq:CVM_statistic}
L_{n}:= L_n\left(\hat{F}_{n}(x), F^{A_{0}, B_{0}}(x)\right) = \int
|\hat{F}_{n}(x)-F^{A_{0}, B_{0}}(x)|^{2}dx,
\end{equation}
where $\hat{F}_n$ is the ESD of $C_n$ defined by (\ref{eq:C_n}) when
$(A_{p},B_{n}) = (A_{0},B_{0})$. Another possible test statistic is a
Cr\'{a}mer-von Mises-type statistic
\begin{equation}\label{eq:CVM_alt_statistic}
V_{n}:= V_n\left(\hat{F}_{n}(x), F^{A_{0}, B_{0}}(x)\right) = \int
|\hat{F}_{n}(x)-F^{A_{0}, B_{0}}(x)|^{2}dF^{A_{0},B_{0}}(x).
\end{equation}
A somewhat similar test for the covariance matrix for cross-sectional data with
real entries with i.i.d. column was proposed by  Pan and Gao \cite{PanG}. In
order to carry out the tests based on the statistic $L_n$ or $V_n$, we need to
obtain the distribution of the test statistics under $H_{0}$. At this point, we
do not have any result on the asymptotic distribution of these test statistics.
However, it can be seen that under $H_0$, both test statistics converge to zero
as $n, p \to \infty$. In this paper, we propose a Monte-Carlo approximation of
the null distribution of $L_{n}$. A similar strategy applies to $V_n$.
Implementing these tests require computing the ESD $\hat{F}_n$ of the matrix
$C_{n0}:=\sqrt{n/p} \left(n^{-1} A_{0}^{1/2}X_{n}B_{0}X_{n}^{*}A_{0}^{1/2}-
\bar{b}^0 A_{0}\right)$, which is obtained by setting $A_{p} = A_{0}$ and $B_n
= B_{0}$ in the definition of $C_n$ in (\ref{eq:C_n}), and where $X_n$ is
chosen to have i.i.d. $N(0,1)$ entries. Notice that, $H_{0}$ does not specify
$B_{0}$ completely, but only specifies its first two spectral moments. Thus,
while carrying out this simulation, we need to construct an appropriate $B_{0}$
whose first two spectral moments are $\bar{b}^0$ and $\bar{b}_2^0$
respectively. We construct $B_{0}$ of the form (assuming, for simplicity, $n$
to be even)
$$
B_{0} = \begin{bmatrix} \beta_1 I_{n/2} & \mathbf{0} \\
\mathbf{0} & \beta_2 I_{n/2} \\
\end{bmatrix}
$$
and solve the equations $\bar{b}^0 = 0.5 \beta_1  + 0.5 \beta_2$ and
$\bar{b}_2^0 = 0.5 \beta_1^2  + 0.5 \beta_2^2$ to obtain $\beta_1$ and
$\beta_2$. In Section \ref{sec:simulation} we conduct a simulation study which
shows that the the histogram of the LSD of $C_{n0}$ is very close to the
theoretical density function  \ref{eq:density} of the LSD $F^{A_{0},B_{0}}(x)$
under $H_0$. In addition, the distribution of $L_{n}$ under $H_{0}$ and $H_{1}$
are well-separated as $p, n \to \infty$ and $p/n \to 0$. We do not present
simulation results involving the statistic $V_{n}$ due to space constraints,
even though the qualitative behavior is similar.

Even though the proposed procedure does not test the separability of the
covariance matrix of the data, we comment on the possibility of extending this
test procedure to deal with some special cases of the latter scenario. The
implementation of these is beyond the scope of this paper.  The corresponding
null hypothesis for the test of separability would be: $H_0 : \Sigma_{n} =
A_{p} \otimes B_{n}$ where $A_{p}$ and $B_{n}$ are unknown positive
semi-definite matrices satisfying that the ESD $F^{A_p}$ converges to a
distribution non-degenerate at zero, and $n^{-1} \tr(B_n) \to 1$ and $n^{-1}
\tr(B_n^2) \to \bar{b}_2$ for some $\bar{b}_2 \geq 1$. The requirement $n^{-1}
\tr(B_n) \to 1$ is to ensure identifiability. In this case, under certain
special structural assumptions on $A_{p}$, it may still be possible to obtain
fairly accurate estimates of $A_{p}$ and $\bar{b}_2$, which can then be used in
the expression for $L_n$ or $V_n$ in place of $A_{0}$ and $\bar{b}_2^{0}$ to
construct a test for separability.  One typical assumption in spatio-temporal
statistics is that the process is stationary either in space or time. In the
current setting, if we assume that the process is row-stationary, then the
eigenvectors of $A_{p}$ can be well-approximated in a discrete Fourier basis.
If in addition, the corresponding spectrum of $A_{p}$ is piecewise constant,
then we can estimate the spectrum of $A_{p}$ from the data as follows. First we
can perform an orthogonal or unitary transformation of the data in the
(approximate) eigen-basis of $A_{p}$.
Then, we can apply a clustering procedure, and estimate the distinct
eigenvalues as the means of the individual clusters, and assign the
eigenvectors to these clusters according to the cluster membership of the
coordinates of the rotated data matrix.
Another way to broaden the class of models under the null hypothesis is to
remove the specification of $\bar{b}_2$. If either the eigenvalues of $A_{p}$
are known or they can be estimated accurately from the data, subject to some
knowledge about the fourth moment of the entries of the data matrix,
$\bar{b}_2$ can be estimated by making use of the expression for
$\mE(\tr(S_n^2))$ in terms of the first two spectral moments of $A_{p}$ and
$B_{n}$.

If $A_{p}$ is unknown but has a relatively small number of distinct
eigenvalues, those eigenvalues can be estimated as mean or median of the
clusters of eigenvalues of $S_n$, without requiring any knowledge of the
eigenvectors of $A_{p}$, by making use of Theorem
\ref{thm:eigenvalue_fluctuations}.

\subsection{Computation of the density function of the
LSD}\label{subsec:computation_density}

If $F^{A}$ is a finite mixture of point masses, then the density function of
the LSD $F$ in Theorem \ref{thm:main_LSD} can be computed numerically by making
use of Proposition \ref{prop:LSD_density} and Lemma
\ref{lemma:density_support}. This computation is used to simulate the
distribution of the test statistic $L_n$ in Section \ref{sec:simulation}.
According to \ref{prop:LSD_density}, the main ingredient of the computation of
$f$, the p.d.f. of $F$, is the determination of the function $\beta(x) :=
\lim_{z\in \mathbb{C}^+ \to x} \beta(z)$ which solves the equation
(\ref{eq:beta_x_equation}). When $F^{A}$ is a finite mixture of point masses,
the latter reduces to a polynomial in $\beta(x)$. In order to determine $f(x)$,
we need to isolate the roots that satisfy the constraints $\Im\beta(x)\geq 0$,
$x\Re\beta(x)<0$ and  $\omega(x)\leq 1$, where $\omega(x)$ is given by
(\ref{eq:omega_x_equation}), as stated in Proposition \ref{prop:LSD_density}.
Indeed, the support of $\beta(x)$ can be determined by the condition $\omega(x)
= 1$, as stated in Lemma \ref{lemma:density_support}. In practice, we
numerically solve for the appropriate root of $\beta(x)$ for a grid of points
$x$ by searching through all possible solutions of the polynomial satisfied by
$\beta(x)$ and checking the conditions, as well as making use of the continuity
of $\beta(x)$ on each side of the origin. Then we can derive the density
function $f(x)$ by utilizing (\ref{eq:density}).



\section{Proof of Theorem \ref{thm:main_LSD}}

Our approach for proving Theorem \ref{thm:main_LSD} is to first restrict to
Gaussian observations and utilize the rank-one perturbation method used in Bai
and Yin \cite{BY} and Pan and Gao \cite{PanG}. However, the decompositions
under the separable case require a slightly different treatment from the
aforementioned references. The extension of the result to non-Gausssian settings is handled in Section
\ref{subsec:nongaussian} through through a use of the Lindeberg principle (see
Chatterjee \cite{Chatterjee}). Another potential route to prove this result is
through the generalized Stein's equations used in Pastur and Shcherbina
\cite{PasturS} and Bao \cite{Bao}.

\subsection{Truncation of the ESD of $A_{p}$}\label{subsec:truncation_F_A}
We begin with a truncation of the spectrum of $A_{p}$. Let $b_0$ be a positive
number such that $\|B_{n}\|\leq b_0$. Define $\bar{b}_n = \tr(B_{n})/n$ and
suppose that $\bar{b}_n \to \bar{b}$ as $n \to \infty$. Let $a_{0} > 0$ be such
that $a_{0}$ is a continuity point of $F^{A}$, the LSD of $A_{p}$. Let
\begin{equation*}
\Lambda_{a_{0}}=\mbox{diag}\left(\lambda_{1}I_{\{\lambda_{1}\leq a_0\}},
\lambda_{2}I_{\{\lambda_{2}\leq a_0\}}, \cdots, \lambda_{p}I_{\{\lambda_{p}\leq
a_0\}}\right).
\end{equation*}
Since $A_{p} = U^{*}\Lambda U$ with $\Lambda=\mbox{diag}(\lambda_{1}, \lambda_{2},
\cdots, \lambda_{p})$, then defining
$\widetilde{A}_{a_0}:=U^{*}\Lambda_{a_0}U$,
$\widetilde{A}_{a_0}^{1/2}:=U^{*}\Lambda_{a_0}^{1/2}U$, and $\widetilde C_{n0}
:=\sqrt{n/p}\left(n^{-1}A_{p}^{1/2}X_{n}B_{n}X_{n}^{*}A_{p}^{1/2}-n^{-1}\tr(B_{n})\widetilde{A}_{a_0}\right)$,
we have
\begin{eqnarray*}
&&\sup_{x}\left|F^{C_{n}}(x)-F^{\widetilde C_{n0}}(x)\right|\\
&&\leq \frac{\bar{b}_{n}}{p}\mbox{Rank}\left(A_{p}-\widetilde{A}_{a_0}\right) =
\frac{\bar{b}_{n}}{p}\sum_{i=1}^{p}I_{\{\lambda_{i}> a_0\}}\rightarrow
\bar{b}\left(1-F^{A}(a_0)\right).
\end{eqnarray*}
Further, defining $\widehat{C}_{n0} :=
\sqrt{\frac{n}{p}}\left(\frac{1}{n}\widetilde{A}_{a_0}^{1/2}X_{n}B_{n}X_{n}^{*}
\widetilde{A}_{a_0}^{1/2}-\frac{\tr(B_{n})}{n}\widetilde{A}_{a_0}\right)$, we
have
\begin{eqnarray*}
&&\sup_{x}\left|F^{\widetilde C_{n0}}(x) - F^{\widehat{C}_{n0}}(x)\right|\\
&&\leq \frac{2}{p} \mbox{Rank}(A_{p}^{1/2}-\widetilde{A}_{a_0}^{1/2})=
\frac{2}{p}\sum_{i=1}^{p} I_{\{\lambda_{i}>a_0\}}\to 2(1-F^{A}(a_0)).
\end{eqnarray*}
By choosing $a_0$ to be large enough, $1-F^{A}(a_0)$ can be made arbitrarily
small. Thus, combining the above two inequalities, we can show that, for any
given $\epsilon > 0$, there exists a large enough $a_0$ such that
\begin{equation*}
\lim\sup_{n\to \infty}
\sup_{x}\left|F^{C_{n}}(x)-F^{\widehat{C}_{n0}}(x)\right| \leq \epsilon \qquad
\mbox{a.s.}
\end{equation*}
Also, in Section \ref{subsec:uniqueness}, we show that the solution of
(\ref{equation_system}) is unique and has a continuous dependence on $F^{A}$.
Thus, since $F^{\widetilde{A}_{a_0}}$ converges to $F^{A}$ in distribution as
$a_0 \to \infty$, in order to prove Theorem \ref{thm:main_LSD}, it is enough to
show that $F^{\widehat{C}_{n0}}$ converges almost surely to $F$, and $F$ has
the Stieltjes transform $s(z)$ determined by (\ref{equation_system}) with
$F^{A} = F^{\widetilde{A}_{a_0}}$, for any fixed positive $a_0$ so that
$F^{\widetilde{A}_{a_0}}$ is not degenerate at zero. For notational
convenience, henceforth, we still use $A_{p}$ and $C_{n}$ instead of
$\widetilde{A}_{a_0}$ and $\widehat{C}_{n0}$, respectively, and simply assume
that of $\parallel A_{p} \parallel \leq a_0$ for an arbitrary positive constant
$a_0$.

\subsection{Expected Stieltjes transforms}\label{subsec:deterministic_part}
In this subsection, we derive asymptotic expansion for $\mathbb{E}(s_n(z))$
when $X_n$ is assumed to have i.i.d. standard normal entries. Let
$A_{p}= U^{*}\Lambda U$ with
$\Lambda=\mbox{diag}(\lambda_{1}, \lambda_{2},\cdots, \lambda_{p})$ denoting the
spectral decomposition of $A_{p}$. Then we have
\begin{eqnarray*}
C_{n}&=&\sqrt{\frac{n}{p}}\left(\frac{1}{n}U^{*}\Lambda^{1/2}UX_{n}B_{n}X_{n}^{*}U^{*}\Lambda^{1/2}U-\frac{1}{n}\tr(B_{n})U^{*}\Lambda U\right)\\
&=& \sqrt{\frac{n}{p}}U^{*}\left(\frac{1}{n}\Lambda^{1/2}UX_{n}B_{n}X_{n}^{*}U^{*}\Lambda^{1/2}-\frac{1}{n}\tr(B_{n})\Lambda\right)U\\
&=& \sqrt{\frac{n}{p}}U^{*}\left(\frac{1}{n}\Lambda^{1/2}\widetilde{X}_{n}B_{n}\widetilde{X}_{n}^{*}\Lambda^{1/2}- \frac{1}{n}\tr(B_{n})\Lambda\right)U\\
&=&\sqrt{\frac{n}{p}}U^{*}(VV^{*}-\bar{b}_{n}\Lambda)U
\end{eqnarray*}
where $\widetilde{X}_{n}= U X_{n}$ and
$V=n^{-1/2}\Lambda^{1/2}\widetilde{X}_{n}B_{n}^{1/2}$. Let $v_{k}=
n^{-1/2}\sqrt{\lambda_{k}}B_{n}^{1/2}\widetilde{x}_{k}$ denote the $k$-th
column of $V^*$, where $\widetilde{x}_{k}$ is the $k$-th column of
$\widetilde{X}_{n}^{*}.$ Note that $\widetilde{X}_{n}$ has i.i.d Gaussian
entries with mean zero and variance one. Moreover, denote by
$\widetilde{X}_{(k)}$ the matrix obtained from $\widetilde{X}_{n}$ with the
$k$-th row replaced by zero. Then
\begin{equation*}
V_{(k)}=(v_{1}^{*},\cdots, v_{k-1}^{*}, 0, v_{k+1}^{*},\cdots, v_{p}^{*})^{*}.
\end{equation*}
We introduce the following quantities:
\begin{eqnarray*}
\omega_{k} &=& \sqrt{\frac{n}{p}}V_{(k)}v_{k}\\
\tau_{kk} &=& \sqrt{\frac{n}{p}}(v_{k}^{*}v_{k}-\bar{b}_{n}\lambda_{k}), \\
Y(z) &=& \sqrt{\frac{n}{p}}\left(VV^{*}-\bar{b}_{n}\Lambda\right)-zI, \\
Y_{k}(z) &=& \sqrt{\frac{n}{p}}\left(V_{(k)}V^{*}-\bar{b}_{n}\Lambda_{(k)}\right)-zI, \\
Y_{(k)}(z) &=& \sqrt{\frac{n}{p}}\left(V_{(k)}V_{(k)}^{*}-\bar{b}_{n}\Lambda_{(k)}\right)-zI,\\
h_{k} &=& \omega_{k}+\tau_{kk}e_{k}.
\end{eqnarray*}
Where $e_{k}$ is supposed to be a canonical unit $p\times 1$ vector with the $k$-th element being $1$ and all others $0$.
Then, notice that $Y(z) = \sum_{k=1}^{p}e_{k}h_{k}^{*} = Y_k(z) +
e_{k}h_{k}^{*}$, for all $k$. Thus, $C_{n}=U^{*} Y(z) U =
U^{*}\left(\sum_{k=1}^{p}e_{k}h_{k}^{*}\right)U$. Since,
$(C_{n}-zI)^{-1}=-z^{-1}I+z^{-1}C_{n}(C_{n}-zI)^{-1}$, we have
\begin{eqnarray}\label{expression for ST}
s_{n}(z)&=& \frac{1}{p}\tr(C_{n}-zI)^{-1} \nonumber\\
&=& \frac{z^{-1}}{p}\tr\left(C_{n}(C_{n}-zI)^{-1}\right)-z^{-1} \nonumber\\
&=&\frac{z^{-1}}{p}\tr\left(\sum_{k=1}^{p}e_{k}h_{k}^{*}\left(\sqrt{\frac{n}{p}}(VV^{*}-\bar{b}_{n}\Lambda)-zI\right)^{-1}\right)-z^{-1}\nonumber\\
&=&\frac{z^{-1}}{p}\tr\left(\sum_{k=1}^{p}e_{k}h_{k}^{*}\left(Y_{k}(z) + e_{k}h_{k}^{*}\right)^{-1}\right)-z^{-1}\nonumber\\
&=&\frac{z^{-1}}{p}\tr\left(\sum_{k=1}^{p}e_{k}h_{k}^{*}\left(Y_{k}^{-1}(z)-\frac{Y_{k}^{-1}(z)e_{k}h_{k}^{*}Y_{k}^{-1}(z)}{1+h_{k}^{*}Y_{k}^{-1}(z)e_{k}}\right)\right)-z^{-1}\nonumber\\
&=&\frac{z^{-1}}{p}\sum_{k=1}^{p}\frac{h_{k}^{*}Y_{k}^{-1}(z)e_{k}}{1+h_{k}^{*}Y_{k}^{-1}(z)e_{k}}-z^{-1}\nonumber\\
&=&-\frac{z^{-1}}{p}\sum_{k=1}^{p}\frac{1}{1+h_{k}^{*}Y_{k}^{-1}(z)e_{k}}
\end{eqnarray}
From the structure of $Y_{k}(z), Y_{{(k)}}(z)$ and $\omega_{k}$, we observe
that
\begin{equation}\label{eq:Y_identity}
Y_{k}^{-1}(z)=\left(Y_{(k)}(z)+\omega_{k}e_{k}^{T}\right)^{-1},~~~
Y_{(k)}^{-1}(z)e_{k}=-z^{-1} e_{k},~~~ ~\omega_{k}^*e_{k}=e_{k}^{T}\omega_{k} =
\sqrt{\frac{n}{p}}e_{k}^{T}V_{(k)}v_{k}=0.
\end{equation}
For any non-negative definite $p\times p$ Hermitian matrix $D$, define
$\underline{D}=UDU^{*}=\{\underline{d}_{ij}\}$. Then it follows that
\begin{eqnarray}\label{plug in formula}
h_{k}^{*}Y_{k}^{-1}(z)\underline{D}e_{k} &=&
h_{k}^{*}\left(Y_{(k)}^{-1}(z)-\frac{Y_{(k)}^{-1}(z)\omega_{k}e_{k}^{T}Y_{(k)}^{-1}(z)}
{1+e_{k}^{T}Y_{(k)}^{-1}(z)\omega_{k}}\right)\underline{D}e_{k}\nonumber\\
&=&h_{k}^{*}\left(Y_{(k)}^{-1}(z) + \frac{1}{z}Y_{(k)}^{-1}(z)\omega_{k}e_{k}^{T}\right)\underline{D}e_{k}\nonumber\\
&=&h_{k}^{*}Y_{(k)}^{-1}(z)\underline{D}e_{k} + \frac{1}{z}h_{k}^{*}Y_{(k)}^{-1}(z)\omega_{k}e_{k}^{T}\underline{D}e_{k}\nonumber\\
&=&\frac{1}{z}\left[\underline{d}_{kk}\omega_{k}^{*}Y_{(k)}^{-1}(z)\omega_{k}-\tau_{kk}\underline{d}_{kk}\right]
+ \omega_{k}^{*}Y_{(k)}^{-1}(z)\underline{D}e_{k}\nonumber\\
&:=& \Rmnum{1}+\Rmnum{2}.
\end{eqnarray}
When $D=I$, the term $\Rmnum{2}$ in (\ref{plug in formula}) equals zero since
by (\ref{eq:Y_identity}),
$\omega_{k}^*Y_{(k)}^{-1}(z)e_{k}=-z^{-1}\omega_{k}^*e_{k}=0$. Plugging this
into (\ref{expression for ST}), and using $\underline{d}_{kk} = 1$, we get
\begin{equation}\label{final equation for expression of ST}
s_{n}(z)=-\frac{1}{p}\sum_{k=1}^{p}\frac{1}{z-\tau_{kk}+\omega_{k}^{*}Y_{(k)}^{-1}(z)\omega_{k}}.
\end{equation}
Moreover, with the expression given by (\ref{expression for ST}) and (\ref{plug
in formula}), we similarly have
\begin{eqnarray}\label{eq:trace_resolvent_D}
\frac{1}{p}\tr\left((C_{n}-zI)^{-1}D\right)
&=& \frac{z^{-1}}{p}\tr\left(C_{n}(C_{n}-zI)^{-1}D\right)-\frac{z^{-1}}{p} \tr(D) \nonumber\\
&=&\frac{z^{-1}}{p}\sum_{k=1}^{p}\frac{h_{k}^{*}Y_{k}^{-1}(z)\underline{D}e_{k}}{1+h_{k}^{*}Y_{k}^{-1}(z)e_{k}}
-\frac{z^{-1}}{p}\tr(\underline{D}) \nonumber\\
&=&\frac{z^{-1}}{p}\sum_{k=1}^{p}\frac{\underline{d}_{kk}\omega_{k}^{*}Y_{(k)}^{-1}(z)\omega_{k}-\tau_{kk}\underline{d}_{kk}}
{z-\tau_{kk}+\omega_{k}^{*}Y_{(k)}^{-1}(z)\omega_{k}}-\frac{z^{-1}}{p}\tr(\underline{D}).
\end{eqnarray}
Define
\begin{equation}\label{eq:beta_n_def}
\beta_{n}(z)=\frac{1}{p}\tr\left((C_{n}-zI)^{-1}A_{p}\right), ~~~z\in
\mathbb{C}^+.
\end{equation}
When $D=A_{p}$, so that $\underline{D}=\Lambda=diag(\lambda_{1}, \lambda_{2},
\cdots, \lambda_{p})$, from (\ref{eq:trace_resolvent_D}), we get
\begin{eqnarray}\label{eq:beta_n_repr}
\beta_{n}(z)&=&\frac{z^{-1}}{p}\tr\left(\sum_{k=1}^{p}\frac{h_{k}^{*}Y_{k}^{-1}(z)\Lambda
e_{k}}{1+h_{k}^{*}Y_{k}^{-1}(z)e_{k}}\right)
-\frac{z^{-1}}{p}\tr(\Lambda)\nonumber\\
&=&
\frac{z^{-1}}{p}\sum_{k=1}^{p}\frac{\lambda_{k}\omega_{k}^{*}Y_{(k)}^{-1}(z)\omega_{k}-\tau_{kk}\lambda_{k}}
{z-\tau_{kk}+\omega_{k}^{*}Y_{(k)}^{-1}(z)\omega_{k}}-\frac{z^{-1}}{p}\sum_{k=1}^{p}\lambda_{k}\nonumber\\
&=&
-\frac{1}{p}\sum_{k=1}^{p}\frac{\lambda_{k}}{z-\tau_{kk}+\omega_{k}^{*}Y_{(k)}^{-1}(z)\omega_{k}}~.
\end{eqnarray}
In order to  derive explicit expressions for $s_{n}(z)$ and $\beta_{n}(z)$, we
still need a further approximation of
$\omega_{k}^{*}Y_{(k)}^{-1}(z)\omega_{k}$. Indeed,
\begin{eqnarray}\label{eq:quadratic term}
\omega_{k}^{*}Y_{(k)}^{-1}(z)\omega_{k}&=&\frac{n}{p}v_{k}^{*}V_{(k)}^{*}Y_{(k)}^{-1}(z)V_{(k)}v_{k} \nonumber\\
&=&\frac{\lambda_{k}}{p}\tr\left(V_{(k)}B_{n}^{1/2}\widetilde{x}_{k}\widetilde{x}_{k}^{*}B_{n}^{1/2}V_{(k)}^{*}Y_{(k)}^{-1}(z)\right) \nonumber\\
&=&\frac{\lambda_{k}}{p}\tr\left(V_{(k)}B_{n}V_{(k)}^{*}Y_{(k)}^{-1}(z)\right) + d_{k}^{(1)} \nonumber\\
&=&\frac{\lambda_{k}}{pn}\tr\left(\Lambda^{1/2}\widetilde{X}_{(k)}B_{n}^{2}\widetilde{X}_{(k)}^{*}\Lambda^{1/2}Y_{(k)}^{-1}(z)\right) + d_{k}^{(1)}\nonumber\\
&=&\frac{\lambda_{k}}{pn}\sum_{i,j\neq k}\left(\sqrt{\lambda_{i}\lambda_{j}}\widetilde{x}_{i}^{*}B_{n}^{2}\widetilde{x}_{j}\right) \left(Y_{(k)}^{-1}(z)\right)_{ji} + d_{k}^{(1)}\nonumber\\
&=&\frac{\lambda_{k}\tr(B_{n}^{2})}{pn}\sum_{i\neq k}\left(\lambda_{i}(Y_{(k)}^{-1}(z))_{ii}\right) + d_{k}^{(2)}\nonumber\\
&=&\bar{b}_{2}(n)\frac{\lambda_{k}}{p}\tr\left(Y_{(k)}^{-1}(z)\Lambda_{(k)}\right)
+d_{k}^{(2)}
\end{eqnarray}
where
$\mE(d_{k}^{(1)})=0$ and $\mE(d_{k}^{(2)})=0$. Note that
\begin{equation*}
\frac{1}{p}\tr\left(Y_{(k)}^{-1}(z)\Lambda_{(k)}\right)\approx
\frac{1}{p}\tr\left(Y^{-1}(z)\Lambda\right) =
\frac{1}{p}\tr\left((C_{n}-zI)^{-1}A_{p}\right)=\beta_{n}(z),
\end{equation*}
which shows that the term $\omega_k^* Y_{(k)}^{-1}(z) \omega_k$ in
(\ref{eq:beta_n_repr}) can be approximated by $\bar b_2(n) \lambda_k
\mE(\beta_n(z))$. Hence we can derive convenient representations for
$\mE(s_{n}(z))$ and $\mE(\beta_{n}(z))$ though this approximation and show that
the remainder terms are negligible.

Let
\begin{equation*}
\epsilon_{k}
:=\bar{b}_{2}(n)\lambda_{k}\mE(\beta_{n}(z))-\omega_{k}^{*}Y_{(k)}^{-1}(z)\omega_{k}+\tau_{kk}.
\end{equation*}
Then, from (\ref{eq:beta_n_repr}) we can write
\begin{equation*}
\beta_{n}(z)=\frac{1}{p}\sum_{k=1}^{p}\frac{\lambda_{k}}{-z-\bar{b}_{2}(n)\lambda_{k}\mE(\beta_{n}(z))
+\epsilon_{k}}.
\end{equation*}
Taking expectation on both sides,
\begin{equation}\label{expectation equation of beta}
\mE(\beta_{n}(z))=\frac{1}{p}\sum_{k=1}^{p}\frac{\lambda_{k}}{-z-\bar{b}_{2}(n)\lambda_{k}\mE(\beta_{n}(z))}+\delta_{n}\\
\end{equation}
where
\begin{equation}\label{eq:delta_n_def}
\delta_{n} = -\frac{1}{p}\sum_{k=1}^{p}\mE\left(\frac{\lambda_{k}
\epsilon_{k}}{\left(-z-\bar{b}_{2}(n)\lambda_{k}\mE(\beta_n(z))\right)
\left(-z-\bar{b}_{2}(n)\lambda_{k}\mE(\beta_{n}(z))+\epsilon_{k}\right)}\right).
\end{equation}
By (\ref{expectation equation of beta}),  to show the convergence of the
expected Stieltjes transform to $\beta(z)$ satisfying (\ref{equation_system}),
it suffices to show that $\delta_{n}\to 0$.
Rewrite $\delta_{n}$ as
\begin{eqnarray*}
\delta_{n} &=&-\frac{1}{p}\sum_{k=1}^{p}\frac{\lambda_{k}\mE(\epsilon_{k})}
{\left(z+\bar{b}_{2}(n)\lambda_{k}\mE(\beta_{n}(z))\right)^{2}} \nonumber\\
&& ~~+ \frac{1}{p}\sum_{k=1}^{p}\mE
\left(\frac{\lambda_{k}\epsilon_{k}^{2}}{\left(z+\bar{b}_{2}(n)\lambda_{k}\mE(\beta_{n}(z))\right)^{2}
\left(-z-\bar{b}_{2}(n)\lambda_{k}\mE(\beta_{n}(z))+\epsilon_{k}\right)}\right)\\
&=& d_{1}+d_{2}.
\end{eqnarray*}
First, by (\ref{eq:quadratic term}), the fact that $\mathbb{E}(d_k^{(2)}) = 0$,
and (\ref{eq:Y_identity}),
\begin{eqnarray}\label{eq:E_epsilon_k_bound}
|\mE(\epsilon_{k})| &=& \frac{\bar{b}_{2}(n)\lambda_{k}}{p}\left|\mE
\left(\tr\left[(C_{n}-zI)^{-1}\Lambda\right]-
\mE\tr\left[Y_{(k)}^{-1}(z)\Lambda_{(k)}\right]\right)\right| \nonumber\\
&=& \frac{\bar{b}_{2}(n)\lambda_{k}}{p}\left|\mE\left( \tr\left[Y^{-1}(z)\Lambda\right]
-\tr\left[Y^{-1}_{(k)}(z)(\Lambda-\lambda_{k}e_{k}e_{k}^{T})\right]\right) \right| \nonumber\\
&=& \frac{\bar{b}_{2}(n)\lambda_{k}}{p}\left|\mE
\tr\left[\left((Y^{-1}(z)-Y_{(k)}^{-1}(z)\right)\Lambda\right]
+ \lambda_{k} \mE  \left(e_{k}^{T} Y_{(k)}^{-1}(z)e_{k}\right)\right| \nonumber\\
&\leq& \frac{\bar{b}_{2}(n)\lambda_{k}}{p}\mE\left|\tr\left[\left((Y^{-1}(z)-Y_{(k)}^{-1}(z)\right)\Lambda\right] \right|
+\frac{\bar{b}_{2}(n)\lambda_k^2}{p|z|} \nonumber\\
&\leq& \frac{M}{p}~,
\end{eqnarray}
which follows from the fact that
\begin{equation}\label{eq:Y_eh_Y_lambda_bound}
\frac{1}{p}\mE\left|\tr\left[\left((Y^{-1}(z)-Y_{(k)}^{-1}(z)\right)\Lambda\right]\right|\leq
\frac{M}{p},
\end{equation}
(see Appendix \ref{attach_estimation}) and that
$(\bar{b}_{2}(n)\lambda_k^2)/(p|z|)\leq M/p$ since $\max_k |\lambda_k| \leq
a_0$ and $\bar{b}_{2}(n) \to \bar{b}_2 < \infty$. Note that
\begin{equation*}
\left|z+\bar{b}_{2}(n)\lambda_{k}\mE \beta_n(z)\right|\geq \Im
(z+\bar{b}_{2}(n)\lambda_{k}\mE \beta_{n}(z))\geq
v+\lambda_k\bar{b}_{2}(n)\mE(\Im \beta_{n}(z))\geq v.
\end{equation*}
Thus, combining with (\ref{eq:E_epsilon_k_bound}), we conclude that $|d_{1}|\to
0$ as $n \to \infty$.

On the other hand,
\begin{equation*}
\left|-z-\bar{b}_{2}(n)\lambda_{k}\mE
\beta_n(z)+\epsilon_{k}\right|=\left|-z-\omega_{k}^{*}Y_{(k)}^{-1}(z)\omega+\tau_{kk}\right|\geq
\Im(z+\omega_{k}^{*}Y_{(k)}^{-1}(z)\omega_{k})\geq v.
\end{equation*}
Hence, to derive $d_{2}\to 0,$ we only need to prove $\mE |\epsilon_{k}|^2\to
0$.  Let
\begin{eqnarray*}
d_{3}&=&\mE|\epsilon_{k}-\mE\epsilon_{k}|^{2}\\
&=&\mE\left|-\omega_{k}^{*}Y_{(k)}^{-1}(z)\omega_{k}+\tau_{kk}-\mE  \omega_{k}^{*}Y_{(k)}^{-1}(z)\omega_{k}\right|^2\\
&\leq& \mE\left|-\omega_{k}^{*}Y_{(k)}^{-1}(z)\omega_{k}+\tau_{kk}-\bar{b}_{2}(n)\frac{\lambda_{k}}{p} tr\left[Y^{-1}_{(k)}(z)\Lambda_{(k)}\right]\right|^2\\
&+&\mE\left|\bar{b}_{2}(n)\frac{\lambda_{k}}{p}tr\left[Y^{-1}_{(k)}(z)\Lambda_{(k)}\right]-\bar{b}_{2}(n)\frac{\lambda_{k}}{p}\mE tr\left[Y^{-1}_{(k)}(z)\Lambda_{(k)}\right]\right|^2\\
&:=& d_{31}+d_{32}
\end{eqnarray*}
where $d_{31}\leq M/p$ and $d_{32}\leq M/p^{3/2}$ (see Appendix \ref{d_31} and
\ref{d_32} for details). Then, we can conclude that $\mE |\epsilon_{k}|^{2}\to
0$ based on the fact that $\mE |\epsilon_{k}|^{2} = \mE|\epsilon_{k}-\mE
(\epsilon_{k})|^2+|\mE (\epsilon_{k})|^2$.

Repeating the same arguments, we can derive the following equation for
$\mE(s_{n}(z))$ given by
\begin{equation}\label{eg:expectation_s_n_z}
\mE(s_{n}(z))=\frac{1}{p}\sum_{k=1}^{p}\frac{1}{-z-\bar{b}_{2}\lambda_{k}\mE(\beta_{n}(z))}+\delta_{n}^{'}
\end{equation}
where $\delta_{n}^{'} \to 0$ as $n \to \infty$.


\subsection{Convergence of $s_{n}(z)-\mE(s_{n}(z))$}\label{subsec:random part}

We proceed to the almost sure convergence of the random parts, i.e., for any
$z\in\mC^{+}$
\begin{equation}\label{ESystem}
\begin{cases}
s_{n}(z)-\mE(s_{n}(z))\stackrel{a.s.}{\longrightarrow} 0\\
\beta_{n}(z)-\mE(\beta_{n}(z))\stackrel{a.s.}{\longrightarrow} 0.
\end{cases}
\end{equation}
when the entries of $X_{n}$ are i.i.d. standardized random variables with
arbitrary distributions. To derive above almost sure convergence, we first get
a concentration inequality by using the following lemma (known as McDiarmid's
inequality) and then finish the proof through Borel-Cantelli lemma.

\begin{lemma}\label{lem:McDiarmid}
(McDiarmid inequality \cite{Mcdiarmid} :) Let $X_{1}, X_{2}, \cdots, X_{m}$ be
independent random vectors taking values in $\mathcal{X}$. Suppose that
$f:\mathcal{X}^{k}\to \mR$ is a function of $X_{1}, X_{2}, \cdots, X_{m}$
satisfying $\forall  x_{1}, \cdots, x_{m}, x_{i}^{'},$
$$
|f(x_1,\cdots,x_{i},\cdots, x_{m})-f(x_1, \cdots, x_{i}^{'}, \cdots,
x_{m})|\leq c_{i},
$$
Then for all $\epsilon>0,$
\begin{equation}\label{MC}
\mP\left(|f(x_{1},\cdots, x_{m})-\mE f(x_{1}, \cdots,
x_{m})|>\epsilon\right)\leq 2
\exp\left(-\frac{2\epsilon^{2}}{\sum_{i=1}^{m}c_{i}^{2}}\right).
\end{equation}
\end{lemma}

Even though Lemma \ref{lem:McDiarmid} is applicable to real-valued functions,
we use it to obtain concentration bounds for $s_n(z)$ (respectively,
$\beta_n(z)$) by applying this result separately to the functions $\Re(s_n(z))$
and  $\Im(s_n(z))$. Thus, we treat $C_{n}$ a function of $X_n$. Let the
independent rows of $X_n$ (written as column vectors) be denoted by
$\mbf{x}_1^*, \cdots, \mbf{x}_{p}^*$.
Let
$$
X_{(i)} = X_{n} - e_i e_i^* X_n = X_{n} - e_i \mbf{x}_i^*,
~~~i=1,\ldots,p.
$$
Thus, $X_{(i)}$ is the $p \times n$ matrix obtained by removing the $i$-th row
from $X_n$ and replacing it by zeros. Define
$$
C_{(i)} = \frac{1}{\sqrt{np}} A_{p}^{1/2} X_{(i)} B_{n} X_{(i)}^* A_{p}^{1/2} -
\frac{1}{\sqrt{np}} \tr(B_{n}) A_{p}.
$$
Then,
\begin{eqnarray}\label{eq:C_diff}
C_n &=& C_{(i)} + \frac{1}{\sqrt{np}} A_{p}^{1/2} e_i \mathbf{x}_i^* B_{n}
X_{(i)}^* A_{p}^{1/2} + \frac{1}{\sqrt{np}} A_{p}^{1/2} X_{(i)} B_{n} \mbf{x}_i
e_i^* A_{p}^{1/2} + \frac{1}{\sqrt{np}}
\mbf{x}_i^* B_{n} \mbf{x}_i A_{p}^{1/2} e_i e_i^* A_{p}^{1/2} \nonumber\\
&=& C_{(i)} + \mbf{a}_i \mbf{y}_i^* + \mbf{y}_i \mbf{a}_i^* + w_i \mbf{a}_i
\mbf{a}_i^*,
\end{eqnarray}
where
$$
\mbf{a}_i = A_{p}^{1/2} e_i, \qquad \mbf{y}_i = \frac{1}{\sqrt{np}}
A_{p}^{1/2} X_{(i)} B_{n} \mbf{x}_i, \qquad w_i = \frac{1}{\sqrt{np}}
\mbf{x}_i^* B_{n} \mbf{x}_i.
$$
We would like to use McDiarmid's inequality (Lemma \ref{lem:McDiarmid}) to
obtain bounds for  $|s_n(z) - \mathbb{E}(s_n(z))|$ and $|\beta_n(z) -
\mathbb{E}(\beta_n(z))|$. In this direction, we first obtain an appropriate
bound for $|p^{-1}\tr((C_n - zI)^{-1}H) - p^{-1}\tr((C_{(i)} - zI)^{-1}H)|$
where $H$ is an arbitrary $p\times p$ Hermitian matrix of bounded norm. $H=I_p$
corresponds to $s_n(z)$ and $H = A_{p}$ corresponds to $\beta_n(z)$. As a first
step, observe that we can write
\begin{equation}\label{eq:y_a_matrix}
\mbf{a}_i \mbf{y}_i^* + \mbf{y}_i \mbf{a}_i^* = \mbf{u}_i \mbf{u}_i^* -
\mbf{v}_i \mbf{v}_i^*, \qquad \mbox{where} ~~\mbf{u}_i = \frac{1}{\sqrt{2}}
(\mbf{a}_i  + \mbf{y}_i)~~\mbox{and}~~ \mbf{v}_i = \frac{1}{\sqrt{2}}
(\mbf{a}_i  - \mbf{y}_i).
\end{equation}
Next, define $D_{1i} = C_{(i)} + \mbf{u}_i \mbf{u}_i^*$ and $D_{2i} = D_{1i} -
\mbf{v}_i \mbf{v}_i^*$ so that $D_{1i} = D_{2i} + \mbf{v}_i \mbf{v}_i^*$. Then,
from (\ref{eq:C_diff}), we have $C_n = D_{2i} + w_i \mbf{a}_i \mbf{a}_i^*$.
Therefore,
\begin{eqnarray}\label{eg:difference_resolvent_deleting_kth_row}
&& \tr((C_n - zI)^{-1} H) - \tr((C_{(i)} - z I)^{-1} H) \nonumber\\
&=& \left[\tr((C_n - zI)^{-1}H) - \tr((D_{2i} - z I)^{-1}H) \right] +
\left[\tr((D_{2i} - zI)^{-1}H) - \tr((D_{1i} - z I)^{-1}H)\right]
\nonumber\\
&& ~~~ + \left[\tr((D_{1i} - zI)^{-1}H) - \tr((C_{(i)} - z I)^{-1}H)\right] \nonumber\\
&=& \frac{w_i \mbf{a}_i^* (D_{2i} - z I)^{-1} H (D_{2i} - z I)^{-1}
\mbf{a}_i}{1+ w_i  \mbf{a}_i^* (D_{2i} - z I)^{-1} \mbf{a}_i} -
\frac{\mbf{v}_i^* (D_{1i} - z I)^{-1}H (D_{1i} - z I)^{-1}\mbf{v}_i}{1+
\mbf{v}_i^* (D_{1i} - z I)^{-1} \mbf{v}_i} \nonumber\\
&& ~~~ + \frac{\mbf{u}_i^* (C_{(i)} - z I)^{-1} H (C_{(i)} - z I)^{-1}
\mbf{u}_i}{1+ \mbf{u}_i^* (C_{(i)} - z I)^{-1} \mbf{u}_i}~.
\end{eqnarray}
Since $w_i \geq 0$ and $D_{2i}$ and $C_{(i)}$ are Hermitian matrices, by Lemma
\ref{lemma:quad_bound}, each of the terms in the last expression is bounded by
$\parallel H \parallel/v$ where $v = \Im(z)$. Thus, taking $H = I_p$ and $H =
A_{p}$, respectively, we have the bounds
$$
\left|\frac{1}{p} \tr((C_n - zI)^{-1}) - \frac{1}{p}\tr((C_n' - z
I)^{-1})\right| \leq \frac{6}{pv} =: c_{0,p}
$$
and
$$
\left|\frac{1}{p} \tr((C_n - zI)^{-1}A_{p}) - \frac{1}{p}\tr((C_n' - z
I)^{-1}A_{p})\right| \leq \frac{6a_0}{pv} =: c_{1,p},
$$
where $C_n'$ is obtained from $C_n$ by replacing $\mbf{x}^*_i$, the $i$-th row
of $X$ by an independent copy, $\mbf{x}_i'$, say, for any $i=1,\ldots,p$. Hence
by Lemma \ref{lem:McDiarmid}, we have for any $\epsilon >0,$
\begin{equation}\label{eq:s_n_concentration}
\mathbb{P}(|s_n(z) - \mathbb{E}(s_n(z))| > \epsilon) \leq 4
\exp\left(-\frac{2\epsilon^2}{p c_{0,p}^2}\right) = 4  \exp\left(-\frac{p
v^2\epsilon^2}{18}\right),
\end{equation}
and
\begin{equation}\label{eq:beta_n_concentration}
\mathbb{P}(|\beta_n(z) - \mathbb{E}(\beta_n(z))| > \epsilon) \leq 4
\exp\left(-\frac{2\epsilon^2}{p c_{1,p}^2}\right) = 4  \exp\left(-\frac{p
v^2\epsilon^2}{18a_0^2}\right).
\end{equation}
Thus, by Borel-Cantelli lemma, $s_n(z) - \mathbb{E}(s_n(z))
\stackrel{a.s.}{\longrightarrow}  0$ and $\beta_n(z) - \mathbb{E}(\beta_n(z))
\stackrel{a.s.}{\longrightarrow}  0$ as $p \to \infty$.

\subsection{Existence and uniqueness}\label{subsec:uniqueness}

In this subsection, we prove the existence and uniqueness of a solution to
(\ref{equation_system}) and its continuous dependence on $F^{A}$. Assuming
first that this is established, we show that $\beta_n(z)
\stackrel{a.s.}{\longrightarrow} \beta(z)$ and $s_n(z)
\stackrel{a.s.}{\longrightarrow} s(z)$ for all $z \in \mathbb{C}^+$. Since
$\mE(\beta_n(z))$ is bounded for $z \in \mathbb{C}^+$, by considering any
subsequence such that $\mE(\beta_{n}(z))$ converges, from (\ref{expectation
equation of beta}), using the Dominated Convergence Theorem, we obtain that
$\mathbb{E}(\beta_n(z))$ converges to  $\beta(z)$ satisfying
(\ref{equation_system}). Then by the fact that $\beta_n(z) - \mE(\beta_{n}(z))
\stackrel{a.s.}{\longrightarrow} 0$, we establish the first assertion. Again,
since $\mE(s_n(z))$ is bounded and $\mathbb{E}(\beta_n(z)) \to \beta(z)$, by
the Dominated Convergence Theorem and using (\ref{eg:expectation_s_n_z}),
$\mathbb{E}(s_n(z)) \to s(z)$, which results in the second assertion by
invoking the fact that $s_n(z) - \mE(s_{n}(z)) \stackrel{a.s.}{\longrightarrow}
0$. Note that this completes the proof of Theorem \ref{thm:main_LSD} when the
entries of $X_{n}$ are i.i.d. standard Gaussian.

In order to establish the existence and uniqueness of a solution of
(\ref{equation_system}), first use the equation for $\beta(z)$ to write
\begin{equation}
s(z)= -\int\frac{dF^{A}(a)}{z + \bar{b}_2a \beta(z)} = \int
\frac{dF^{A}(a)}{-z+\bar{b}_{2}a\int\frac{tdF^{A}(t)}{z+\bar{b}_{2}t\beta(z)}}.
\end{equation}
The two sides of the last equality gives the following equivalent
representation of  (\ref{equation_system}):
\begin{equation}\label{equation_system_combined}
\left(\beta(z)+\int \frac{adF^{A}(a)}{z+\bar{b}_{2}a\beta(z)}\right)\left(\int
\frac{\bar{b}_{2}t}{(z+\bar{b}_{2}t\beta(z))(\bar{b}_{2}t\int
\frac{adF^{A}(a)}{z+\bar{b}_{2}a\beta(z)}-z)}dF^{A}(t)\right)=0
\end{equation}
We will show that if $F^{A}$ is not the degenerate distribution at zero, we
have
\begin{equation}\label{uniq}
\int \frac{\bar{b}_{2}t}{(z+\bar{b}_{2}t\beta(z))(\bar{b}_{2}t\int
\frac{adF^{A}(a)}{z+\bar{b}_{2}a\beta(z)}-z)}dF^{A}(t)\ne 0,
\end{equation}
so that there is a solution, and that there is a unique $\beta(z)$ satisfying
(\ref{equation_system_combined}). Let $\beta(z)=\beta_{1}+i\beta_{2}$. In view
of establishing the continuous dependence of $\beta(z)$, and hence $s(z)$, on
$F^{A}$, suppose that there is another distribution $F^{A^{0}}$, also
non-degenerate at zero. And let, $\beta^{0}(z)=\beta^{0}_{1}+i\beta^{0}_{2}\in
\mC^{+}$ satisfies
$$
\beta^{0}(z)=-\int\frac{adF^{A^{0}}(a)}{z+\bar{b}_{2}a\beta^{0}(z)}.
$$
Then we have
\begin{eqnarray}\label{eq:uniqueness_on_beta_n}
\beta(z)-\beta^{0}(z)&=&-\int\frac{adF^{A}(a)}{z+\bar{b}_{2}a\beta(z)}+\int\frac{adF^{A^{0}}(a)}{z+\bar{b}_{2}a\beta^{0}(z)}\nonumber\\
&=&\left(\beta(z)-\beta^{0}(z)\right)\gamma(z)+\int
\frac{a}{z-\bar{b}_{2}a\beta^{0}(z)}d(F^{A}(a)-F^{A^{0}}(a))
\end{eqnarray}
where
$$
\gamma(z):=\int
\frac{\bar{b}_{2}a^{2}dF^{A}(a)}{(z+\bar{b}_{2}a\beta(z))(z+\bar{b}_{2}a\beta^{0}(z))}.
$$
Let
$$
\omega(z)=\int
\frac{\bar{b}_{2}a^{2}}{\left|z+\bar{b}_{2}a\beta(z)\right|^{2}}dF^{A}(a),
\quad \omega^{0}(z)=\int
\frac{\bar{b}_{2}a^{2}}{\left|z+\bar{b}_{2}a\beta^{0}(z)\right|^{2}}dF^{A}(a) ,
$$
$$
\tau(z)=\int \frac{a}{\left|z+\bar{b}_{2}a\beta(z)\right|^{2}}dF^{A}(a) ,\quad
\tau^{0}(z) = \int
\frac{a}{\left|z+\bar{b}_{2}a\beta^{0}(z)\right|^{2}}dF^{A}(a).
$$
We have
\begin{equation}\label{imaginary part of beta}
\beta_{2}=-\Im
\int\frac{a(\overline{z}+\overline{a\beta(z)})}{\left|z+\bar{b}_{2}a\beta(z)\right|^{2}}dF^{A}(a)=\int
\frac{av+\bar{b}_{2}a^{2}\beta_{2}}{\left|z+\bar{b}_{2}a\beta(z)\right|^{2}}dF^{A}(a):=\beta_{2}\omega(z)+v\tau(z)
\end{equation}
and
\begin{equation}\label{eg:expresssion_tau_0}
\beta_{2}^{0}=\int
\frac{av+\bar{b}_{2}a^{2}\beta^{0}_{2}}{\left|z+\bar{b}_{2}a\beta^{0}(z)\right|^{2}}dF^{A^{0}}(a):=\beta^{0}_{2}\omega^{0}(z)+v\tau^{0}(z),
\end{equation}
By Cauchy-Schwarz inequality,
\begin{eqnarray*}
|\gamma(z)| &=&\left|\int \frac{\bar{b}_{2}a^{2}dF^{A}(a)}{(z+a\beta(z))(z+a\beta^{0}(z))}\right|\\
&\leq& \int \left|\frac{\bar{b}_{2}a^{2}}{(z+a\beta(z))^2}\right|^{1/2}\left|\frac{\bar{b}_{2}a^{2}}{(z+a\beta^{0}(z))^2}\right|^{1/2}dF^{A}(a)\\
&\leq& \left[\int
\frac{\bar{b}_{2}a^{2}dF^{A}(a)}{|z+\bar{b}_{2}a\beta(z)|^{2}}\right]^{1/2}
\left[ \int \frac{\bar{b}_{2}a^{2}dF^{A}(a)}{|z+\bar{b}_{2}a\beta^{0}(z)|^{2}}\right]^{1/2}\\
&=&\left(\frac{\beta_{2}\omega(z)}{\beta_{2}\omega(z)+v\tau(z)}\right)^{1/2}
\left(\frac{\beta^{0}_{2}\omega^{0}(z)}{\beta^{0}_{2}\omega_{0}(z)+v\tau^{0}(z)}\right)^{1/2}\\
&< & 1.
\end{eqnarray*}
The last inequality holds is due to the fact that  for $v > 0$,
$$
\beta_{2}=\beta_{2}\omega(z)+v\tau(z)>\beta_{2}\omega(z)
$$
which implies
$$
\omega(z)=\int
\frac{\bar{b}_{2}a^{2}}{\left|z+\bar{b}_{2}a\beta(z)\right|^{2}}dF^{A}(a)<1,
$$
and it also holds that
\begin{equation}\label{eq:omega_0}
\omega^{0}(z)=\int
\frac{\bar{b}_{2}a^{2}}{\left|z+\bar{b}_{2}a\beta^{0}(z)\right|^{2}}dF^{A}(a)<
1.
\end{equation}
From (\ref{eq:uniqueness_on_beta_n}) we have
\begin{eqnarray}\label{eq:difference_beta_beta_zero}
\left|\beta(z)-\beta^{0}(z)\right|&=&\frac{1}{1-\gamma(z)}\left|\int
\frac{a}{z+\bar{b}_{2}a\beta^{0}(z)}d(F^{A}(a)-F^{A^{0}}(a))\right|
\end{eqnarray}
from which the uniqueness of the solution $\beta(z)$ follows.
If $F^{A^{0}}$ is not degenerate at zero, then the integrand
(\ref{eq:difference_beta_beta_zero}) is a bounded (and continuous) function of
$a$, which establishes the continuous dependence of $\beta(z)$ on $F^{A}$
through the characterization of distributional convergence. To see this, note
that $|\beta^0(z)| > 0$ implies that for all $a > M$,
\begin{equation*}
\left|\frac{a}{z + \bar{b}_2 a \beta^{0}(z)}\right| = \frac{1}{| z/a +
\bar{b}_2 \beta^0(z)|} \leq \frac{1}{\bar{b}_2 |\beta^0(z)| - |z|/M}
\end{equation*}
where $M$ is large enough that the denominator in the last expression is
positive. On the other hand, for $0 \leq a \leq M$,
\begin{equation*}
\left|\frac{a}{z + \bar{b}_2 a \beta^{0}(z)}\right| \leq \frac{M}{v}
\end{equation*}
since $\Im(\beta^0(z)) = \beta_2^0 > 0$ by (\ref{eg:expresssion_tau_0}) and
(\ref{eq:omega_0}).


Now, to prove (\ref{uniq}), we write
\begin{equation*}
\left(\int \frac{\bar{b}_{2}t}{(z+\bar{b}_{2}t\beta(z))(\bar{b}_{2}t\int
\frac{adF^{A}(a)}{z+\bar{b}_{2}a\beta(z)}-z)}dF^{A}(t)\right)=\int
\frac{\bar{b}_{2}t g(z)}{|z+\bar{b}_{2}t\beta(z)|^{2}|\bar{b}_{2}t\int
\frac{adF^{A}(a)}{z+\bar{b}_{2}a\beta(z)}-z|^{2}}dF^{A}(t)
\end{equation*}
where {\small
\begin{eqnarray*}
&&g(z)\\
&:=& (\overline{z}+\bar{b}_{2}\overline{t\beta(z)})(\bar{b}_{2}\overline{t\int \frac{adF^{A}(a)}{z+\bar{b}_{2}a\beta(z)}}-\overline{z})\\
&=&[u+\bar{b}_{2}t\beta_{1}-i(v+\bar{b}_{2}t\beta_{2})]
\left[\bar{b}_{2}t\int\frac{a(u+\bar{b}_{2}a\beta_{1})dF^{A}(a)}{|z+\bar{b}_{2}a\beta(z)|^{2}} - u + i\left(\bar{b}_{2}t\int\frac{a(v+\bar{b}_{2}a\beta_{2})dF^{A}(a)}{|z+\bar{b}_{2}a\beta(z)|^{2}}+v\right)\right]\\
&=&\left[
(u+\bar{b}_{2}t\beta_{1})\left(\bar{b}_{2}t\int\frac{a(u+\bar{b}_{2}a\beta_{1})dF^{A}(a)}{|z+\bar{b}_{2}a\beta(z)|^{2}}-u\right)
+(v+\bar{b}_{2}t\beta_{2})\left(\bar{b}_{2}t\int\frac{a(v+\bar{b}_{2}a\beta_{2})dF^{A}(a)}{|z+\bar{b}_{2}a\beta(z)|^{2}}+v\right)\right]\\
&& -
i\left[(u+\bar{b}_{2}t\beta_{1})\left(\bar{b}_{2}t\int\frac{a(v+\bar{b}_{2}a\beta_{2})dF^{A}(a)}{|z+\bar{b}_{2}a\beta(z)|^{2}}+v\right)
-(v+\bar{b}_{2}t\beta_{2})\left(\bar{b}_{2}t\int\frac{a(u+\bar{b}_{2}a\beta_{1})dF^{A}(a)}{|z+\bar{b}_{2}a\beta(z)|^{2}}-u\right)\right]
\end{eqnarray*}
}From the fact that $\Im(\tr((C_n-zI)^{-1}A_{p}))\geq 0$, we have
$\beta_{2}\geq 0$. Then, if $(u+\bar{b}_{2}t\beta_{1})$ and
$$
\left(\bar{b}_{2}t\int\frac{a(u+\bar{b}_{2}a\beta_{1})dF^{A}(a)}{|z+\bar{b}_{2}a\beta(z)|^{2}}-u\right)
$$
are either both nonpositive or both nonnegative, then $\Re g(z)$ is positive.
Else, if these two terms have opposite signs, the imaginary part of $g(z)$ is
non-zero. Therefore (\ref{uniq}) is established.

\vskip.15in
Finally, we give the proof of Lemma \ref{lemma:density_support}
based on the facts given in this subsection.
\begin{proof}\label{support_proof}
By (\ref{imaginary part of beta}), taking $v=\Im(z) = 0$ (since $z = x$), we
have $\Im\beta(x) = \omega(x)\Im\beta(x)$, which shows that, $\Im \beta(x) > 0$
if and only if $\omega(x) = 1$.  This, together with (\ref{eq:density}), shows
that $f(x) \neq 0$ if and only if $x \Re\beta(x) < 0$ and $\omega(x) = 1$,
i.e., if $\Im \beta(x) > 0$.
\end{proof}

\subsection{Extension to non-Gaussian settings}\label{subsec:nongaussian}

In this section, we will prove the Theorem \ref{thm:main_LSD} for non-Gaussian
observations through the Lindeberg's replacement strategy (see Chatterjee
\cite{Chatterjee}). As a first step, we perform a truncation, centering and
rescaling of the entries of $X_{n}$. Let
$\breve{X}_{ij}=X_{ij}I_{\{|X_{ij}|\leq n^{1/4}\epsilon_{p}\}}$ and
$\hat{X}_{ij}=(\breve{X}_{ij}-\mE(\breve{X}_{ij}))/\sqrt{\mbox{Var}(\breve{X}_{ij})}$,
where $\epsilon_{p}$ is chosen to satisfy $\epsilon_{p}\to 0,
\epsilon_{p}p^{1/4}\to \infty$ and $\mathbb{P}(|X_{11}| \geq \epsilon_p
p^{1/4}) \leq \epsilon_p/n$. Define
$$
\hat{C_{n}}:=
\sqrt{\frac{n}{p}}\left(\frac{1}{n}A_{p}^{1/2}\hat{X}_{n}B_{n}\hat{X}_{n}^{*}A_{p}^{1/2}-\frac{1}{n}\tr(B_{n})A_{p}\right),
$$
then according to Bai and Yin \cite{BY}, by applying a rank inequality and the
Bernstein's inequality, we get
$$
\sup_{x}\left|F^{C_{n}}(x)-F^{\hat{C}_{n}}(x)\right|\stackrel{a.s.}{\longrightarrow}
0.
$$
For notational simplicity, the truncated, centered and rescaled variables are
henceforth still denoted by $X_{ij}$ and we henceforth assume that $X_{ij}$'s
are i.i.d. with $|X_{ij}| \leq n^{1/4}\epsilon_p$, $\mathbb{E}(X_{ij}) = 0$,
$\mathbb{E}|X_{ij}|^2 = 1$ and $\mathbb{E}|X_{ij}|^4 \leq C$ for some $C <
\infty$.

Define
\begin{equation*}
\widetilde{C}_{n}=\sqrt{\frac{n}{p}}\left(\frac{1}{n}A_{p}^{1/2}W_{n}B_{n}W_{n}^{*}A_{p}^{1/2}-\frac{\tr(B_{n})}{n}A_{p}\right),
\end{equation*}
where the entries of $W=(W_{ij})_{p\times n}$ are i.i.d Gaussian random
variables with $\mE (W_{11})=0$ and $\mE (|W_{11}|^{2})=1$. Suppose $W_{ij}$
are independent of $X_{ij}$ defined in Theorem \ref{thm:main_LSD}. The key step
is to estimate the difference of
\begin{equation}
\mE\left(\frac{1}{p}\tr(C_{n}-zI)^{-1}\right)-\mE\left(\frac{1}{p}\tr(\widetilde{C}_{n}-zI)^{-1}\right)
\end{equation}
and to show that it converges to 0 as $n, p\to \infty$. In fact, the
Gaussianity of $W_{ij}$ is not used in the proof, only moment conditions on
$W_{ij}$ are required. To apply the Lindeberg principle,  we denote
$$
X_{11}, \cdots, X_{1n}, X_{2n},\cdots, X_{pn}\quad by \quad Y_{1}, \cdots, Y_{pn},
$$
and
$$
W_{11}, \cdots, W_{1n}, W_{2n},\cdots, W_{pn}\quad by \quad \y_{1}, \cdots,   \y_{pn}.
$$
Let $m=m(n)=pn$, for each $1\leq i \leq m,$ define
$$
Z_{i}=(Y_{1}, \cdots,Y_{i-1}, Y_{i}, \y_{i+1}, \cdots, \y_{m})
$$
and
$$
Z_{i}^{0}=(Y_{1}, \cdots, Y_{i-1}, 0, \y_{i+1}, \cdots, \y_{m}).
$$
Suppose that $f: \mR^{m}\to \mC^{+}$ is defined as $f(\mathbf{y}) = p^{-1}
\tr(C(\mathbf{y})-zI)^{-1}$ where $C(\mathbf{y}) = n^{-1}
A_{p}^{1/2}\mathbf{Y}B_{n} \mathbf{Y}^*A_{p}^{1/2}$, where the $p \times n$
matrix $\mathbf{Y}$ is obtained by converting the $m \times 1$ vector
$\mathbf{y}$. Then $f(Z_{m})=p^{-1}\tr(C_{n}-zI)^{-1}$ and $f(Z_{0})
=p^{-1}\tr(\widetilde{C}_{n}-zI)^{-1}$. So we rewrite the difference as
\begin{equation}\label{difference on resolvent}
\mE\left[\frac{1}{p}\tr(C_{n}-zI)^{-1}\right]-\mE\left[\frac{1}{p}\tr(\widetilde{C}_{n}-zI)^{-1}\right]
= \sum_{i=1}^{m}\mE\left[f(Z_{i})-f(Z_{i-1})\right].
\end{equation}
Since $f$ is thrice continuously differentiable, a third order Taylor expansion
yields:
\begin{equation}\label{Taylor1}
f(Z_{i})=f(Z_{i}^{0})+Y_{i}\p_{i}f(Z_{i}^{0})+\frac{1}{2}Y_{i}^{2}\p_{i}^{2}f(Z_{i}^{0})
+ \frac{1}{2}Y_{i}^{3}\int_{0}^{1}(1-t)^{2}\p_{i}^{3}f(Z_{i}^{(1)}(t))dt
\end{equation}
and
\begin{equation}\label{Taylor2}
f(Z_{i-1})=f(Z_{i}^{0})+\y_{i}\p_{i}f(Z_{i}^{0})+\frac{1}{2}\y_{i}^{2}\p_{i}^{2}f(Z_{i}^{0})+
\frac{1}{2}\y_{i}^{3}\int_{0}^{1}(1-t)^{2}\p_{i}^{3}f(Z_{i-1}^{(2)}(t))dt
\end{equation}
where $\p_{i}^{r}$ denotes the $r$-fold partial derivative ($r =1,2,3$) with
respect to the $i$-th coordinate and
$$
Z_{i}^{(1)}(t)=(Y_{1}, \cdots,Y_{i-1}, tY_{i}, \y_{i+1}, \cdots, \y_{m}),
$$
and
$$
Z_{i-1}^{(2)}(t)=(Y_{1}, \cdots, Y_{i-1}, t\y_{i}, \y_{i+1}, \cdots, \y_{m}).
$$
Since both $Y_{i}$ and $\y_{i}$ have zero mean and unit variance and are
independent of $Z_{i}^{0}$, the expectation of first order and second order
terms in (\ref{Taylor1}) and (\ref{Taylor2}) are zero. Thus (\ref{difference on
resolvent}) becomes
\begin{equation*}
\sum_{i=1}^{m}\mE\left[f(Z_{i})-f(Z_{i-1})\right] =
\frac{1}{2}\sum_{i=1}^{m}\mE\left[Y_{i}^{3}\int_{0}^{1}(1-t)^{2}\p_{i}^{3}f\left(Z_{i}^{(1)}(t)\right)dt
-
\y_{i}^{3}\int_{0}^{1}(1-t)^{2}\p_{i}^{3}f\left(Z_{i-1}^{(2)}(t)\right)dt\right].
\end{equation*}

In the following, to avoid complicated notations, unless otherwise specified,
we will use the notation $X_{ij}$ to mean either $X_{ij}$ or $W_{ij}$ since
their role will be only in terms of providing bounds for the expected values of
the remainder terms in the expansion above. $X_{n}$ will be used to denote a
matrix containing the corresponding mixed terms. The properties of these random
variables that we will use are that they are independent, have zero mean and
unit variance, and are sub-Gaussian (bounded in case of $X_{ij}$'s).
Accordingly, let $G_{n}=G_{n}(z):=(C_{n}-zI)^{-1}$. To derive a bound for the
terms involving $\p_{i}^{3}f\left(Z_{i}^{(k)}(t)\right)$, $k=1,2$, we need a
bound on $p^{-1}\tr\left[\frac{\p^{3}G_{n}}{\p X_{ij}^{3}}\right]$. Since
$\frac{\p G_{n}}{\p X_{ij}}=-\frac{\p C_n}{\p X_{ij}}G_{n}^{2}$, we get
\begin{equation}\label{third_derivative}
\frac{1}{p}\tr\left[\frac{\p^3 G_{n}}{\p X^{3}_{ij}}\right] =
\frac{6}{p}\tr\left[\frac{\p C_{n}}{\p X_{ij}}G_{n}\frac{\p^2 C_{n}}{\p
X_{ij}^{2}}G_{n}^{2}\right] - \frac{6}{p}\tr\left[\frac{\p C_{n}}{\p
X_{ij}}G_{n}\frac{\p C_{n}}{\p X_{ij}}G_{n}\frac{\p C_{n}}{\p
X_{ij}}G_{n}^2\right].
\end{equation}
where
\begin{equation*}
\frac{\p C_{n}}{\p X_{ij}}=
\frac{1}{\sqrt{np}}\left(A_{p}^{1/2}X_{n}B_{n}\widetilde{e}_{j}e_{i}^{T}A_{p}^{1/2}
+ A_{p}^{1/2}e_{i}\widetilde{e}_{j}^{T}B_{n}X_{n}^{*}A_{p}^{1/2}\right)
\end{equation*}
and
\begin{equation*}
\frac{\p^{2}C_{n}}{\p
X_{ij}^{2}}=\frac{2}{\sqrt{np}}b_{jj}A_{p}^{1/2}e_{i}e_{i}^{T}A_{p}^{1/2},
\quad \frac{\p^{3}C_{n}}{\p X_{ij}^{3}}=0,
\end{equation*}
in which $e_{i}$ is a $p\times 1$ unit vector and $\widetilde{e}_{i}$ is a
$n\times 1$ unit vector.

Let $ r_{j}:= A_{p}^{1/2}X_{n}B_{n}\widetilde{e}_{j}$ and
$\xi_{i}:=A_{p}^{1/2}e_{i}$. The first term in (\ref{third_derivative}) becomes
\begin{eqnarray}
\frac{6}{p}\tr\left[\frac{\p C_{n}}{\p X_{ij}}G_{n}\frac{\p^2 C_{n}}{\p
X_{ij}^{2}}G_{n}^{2}\right]
&=&\frac{12b_{jj}}{np^2}\left[\xi_{i}^{*}G_{n}^{2}\xi_{i}\xi_{i}^{*}G_{n}r_{j}\right]
+ \frac{12b_{jj}}{np^2}\left[r_{j}^{*}G_{n}\xi_{i}\xi_{i}^{*}G_{n}^{2}\xi_{i}\right]\nonumber\\
&:=& \eta_{1}(n)+\eta_{2}(n)
\end{eqnarray}
and the second term in (\ref{third_derivative}) becomes
\begin{eqnarray*}
\frac{1}{p}\tr\left[\frac{\p C_{n}}{\p X_{ij}}G_{n}\frac{\p C_{n}}{\p
X_{ij}}G_{n}\frac{\p C_{n}}{\p X_{ij}}G^{2}_{n}\right]
&=&\frac{1}{n^{3/2}p^{5/2}} \tr\left[(r_{j}\xi_{i}^{*} +
\xi_{i}r_{j}^{*})G_{n}(r_{j}\xi_{i}^{*}
+ \xi_{i}r_{j}^{*})G_{n}(r_{j}\xi_{i}^{*} + \xi_{i}r_{j}^{*})G_{n}^{2}\right]\\
&:=& 2\eta_{3}(n)+2\eta_{4}(n)+2\eta_{5}(n)+2\eta_{6}(n)
\end{eqnarray*}
where
\begin{eqnarray}
\eta_{3}(n)&=& \frac{1}{n^{3/2}p^{5/2}}\left[(r_{j}^{*}G_{n}\xi_{i})^2r_{j}^{*}G_{n}^{2}\xi_{i}\right]\\
\eta_{4}(n) &=& \frac{1}{n^{3/2}p^{5/2}}\left[r_{j}^{*}G_{n}\xi_{i}r_{j}^{*}G_{n}r_{j}\xi_{i}^{*}G_{n}^{2}\xi_{i}\right]\\
\eta_{5}(n) &=&  \frac{1}{n^{3/2}p^{5/2}}\left[r_{j}^{*}G_{n}r_{j}\xi_{i}^{*}G_{n}\xi_{i}r_{j}^{*}G^{2}_{n}\xi_{i}\right]\\
\eta_{6}(n)&=&
\frac{1}{n^{3/2}p^{5/2}}\left[\xi_{i}^{*}G_{n}\xi_{i}r_{j}^{*}G_{n}\xi_{i}r_{j}^{*}G^{2}_{n}r_{j}\right].
\end{eqnarray}
To complete the proof, we need the following lemma whose proof is given in the
Appendix.
\begin{lemma}\label{lem:estimation_r}
For any positive number $k\geq 1$,
\begin{equation}\label{r}
\mE\|r_{j}\|^{2k}\leq C_{k}p^{k}
\end{equation}
for some positive constant $C_{k}$.
\end{lemma}

To estimate (\ref{difference on resolvent}), we need to find appropriate bounds
for $\mE \left|Y_{i}^{3}\p_{i}^{3}f\left(Z_{i}^{(1)}(t)\right)\right|$ and $\mE
\left|\tilde Y_{i}^{3}\p_{i}^{3}f\left(Z_{i}^{(2)}(t)\right)\right|$. Note that
for $k = 1, 2$,
\begin{equation}\label{eta_1}
|\eta_{k}(n)|\leq \frac{12b_{0}}{np^2}\left(\frac{a_{0}}{v}\right)^3\|r_{j}\|.
\end{equation}
Applying H\"{o}lder's inequality and (\ref{eta_1}) Lemma \ref{lem:estimation_r}
gives, for $k=1,2$,
\begin{equation*}
\mE \left[\left|X_{ij}^3\eta_k(n)\right|\right]\leq \frac{M}{np^2}\left[\mE
|X_{ij}|^{4}\right]^{3/4} \left[\mE\|r_{j}\|^{4}\right]^{1/4}\leq
\frac{M}{np^{3/2}}.
\end{equation*}
For $k= 3,4,5,6$ we have
\begin{equation}\label{eta_3}
\left|\eta_{k}(n)\right|\leq
\frac{1}{n^{3/2}p^{5/2}}\left(\frac{a_{0}}{v}\right)^{3}\|r_{j}\|^{3}.
\end{equation}
Since $|X_{ij}|\leq n^{1/4}\epsilon_{p}$ and (\ref{eta_3}), by using
Cauchy-Schwarz inequality and Lemma \ref{lem:estimation_r} we have
\begin{eqnarray*}
\mE\left[|X_{ij}^{3}\eta_{3}(n)|\right] &\leq& \frac{M}{n^{3/2}p^{5/2}}\mE\left[|X_{ij}|^{3}\|r_{j}\|^{3}\right]\\
&\leq&\frac{M}{n^{3/2}p^{5/2}}\left[\left(\mE|X_{ij}|^{6}\right)^{1/2}\left(\mE \|r_{j}\|^{6}\right)^{1/2}\right]\\
&\leq&\frac{M}{n^{3/2}p}\left[\mE|X_{ij}|^{4}\right]^{1/2}n^{1/4}\epsilon_{p}\\
&\leq& \frac{M\epsilon_{p}}{n^{5/4}p}.
\end{eqnarray*}
Therefore, by applying (\ref{eta_3}) and Lemma \ref{lem:estimation_r}, it also
holds for $k = 4,5,6$ that  $\mE
\left[|X_{ij}^{3}\eta_{k}(n)|\right]=O(\epsilon_{p}n^{-5/4}p^{-1})$. If instead
of $X_{ij}$ the terms involved were $W_{ij}$, we could simply use the fact that
all moments of $W_{ij}$ are finite to reach the same conclusion. Thus,
combining the bounds, (\ref{difference on resolvent}) can be bounded by
\begin{equation*}
\sum_{i=1}^{m}\int_{0}^{1}(1-t)^{2}\left[\mE\left|Y_{i}^{3}\p_{i}^{3}f\left(Z_{i}^{(1)}(t)\right)\right|
+
\mE\left|\y_{i}^{3}\p_{i}^{3}f\left(Z_{i-1}^{(2)}(t)\right)\right|\right]dt\leq
M \max\{n^{-1/4}\epsilon_{p},p^{-1/2} \}\to 0.
\end{equation*}
This completes the proof of Theorem \ref{thm:main_LSD}.


\section{Proof of Theorem \ref{thm:eigenvalue_fluctuations}}\label{sec:proof_eigenvalue_fluctuations}

Without loss of generality, we can take
\begin{equation}\label{eq:definition_A_p}
A_{p}^{1/2} =
\begin{pmatrix}
\sqrt{\alpha_1} I_{p_1} & \mathbf{0} & \cdots & \mathbf{0} \\
\mathbf{0} & \sqrt{\alpha_2} I_{p_2} & \cdots & \mathbf{0} \\
\cdot & \cdot & \cdots & \cdot \\
\mathbf{0} & \mathbf{0} & \cdots & \sqrt{\alpha_m} I_{p_m} \\
\end{pmatrix}
V^* = \begin{pmatrix}
\sqrt{\alpha_1} V_1^* \\
\sqrt{\alpha_2} V_2^* \\
\cdots \\
\sqrt{\alpha_m} V_m^* \\
\end{pmatrix},
\end{equation}
for $\alpha_1 \geq \alpha_2 \cdots \geq \alpha_m$, where $V = [V_1 : \cdots :
V_m]$ is a $p\times p$ unitary matrix where $V_j$ is a $p \times p_j$ matrix,
so that $V_j^* V_j = I_{p_j}$ for $j=1,\ldots,m$ and $V_j^* V_k =
\mathbf{0}_{p_j \times p_k}$ for $1\leq j \neq k \leq m$. Thus, the data matrix
$Y_n$ can be expressed as
\begin{equation}\label{eq:Y_repr_A_discrete}
Y_n = A_{p}^{1/2} X_n B_{n}^{1/2} =
\begin{pmatrix}
\sqrt{\alpha_1} V_1^* X_n B_{n}^{1/2} \\
\cdots \\
\sqrt{\alpha_m} V_m^* X_n B_{n}^{1/2} \\
\end{pmatrix}
= \begin{pmatrix}
Y_1 \\
\cdot \\
Y_m \\
\end{pmatrix}.
\end{equation}
Assume that $B_{n}$ is a $n\times n$ matrix such that $\bar{b}_n :=
n^{-1}\tr(B_n) \to 1$. Then the sample covariance matrix $S_n =
n^{-1}Y_nY_n^{*}$ with mean $\bar{b}_n A_p$ can be expressed as
\begin{equation}
S_n =
\begin{pmatrix}
S_{11} & \cdots&S_{1m}\\
\vdots&\ddots& \vdots\\
S_{m1}& \cdots &S_{mm}\\
\end{pmatrix}
\end{equation}
As a first step, we define the following renormalized matrix
\begin{equation}\label{def:D_n}
D_n := \sqrt{\frac{n}{p}}(\mathbf{D}(S_n)- \bar{b}_n A_p)\quad \text{where}
\quad \mathbf{D}(S_n)=
\begin{pmatrix}
S_{11} & \cdots&\mathbf{0}\\
\vdots&\ddots& \vdots\\
\mathbf{0}& \cdots &S_{mm}\\
\end{pmatrix};
\end{equation}
and
\begin{equation}\label{def:E_n}
E_n := \sqrt{\frac{n}{p}}(\bs{\lambda}(S_n)- \bar{b}_n A_p);\quad
\bs{\lambda}(S_n)= diag(\lambda_1(S_n),\cdots, \lambda_p(S_n)),
\end{equation}
where $\lambda_j(C)$ denotes the $j$-th largest eigenvalue of the Hermitian
matrix $C$. The ESD of $D_n$ converges weakly almost surely to a nonrandom
distribution $F$, where $F(x) = \sum_{j=1}^m c_j F_{sc}(x; \sqrt{\bar{b}_2
c_j} \alpha_j)$, where $c_{j}=p_{j}/p$. This is established by
observing that the Stieltjes transform of $D_n$ can be expressed as
\begin{equation*}
\frac{1}{p}\tr((\sqrt{\frac{n}{p}}(\mathbf{D}(S_n) - \bar{b}_n A_p) - z
I_p)^{-1}) = \sum_{j=1}^m c_j \frac{1}{p_j} \tr((\sqrt{c_j}\alpha_j
\sqrt{\frac{n}{p_j}}(S_{jj}/\alpha_j  - \bar{b}_n I_{p_j}) - z I_p)^{-1}),
\end{equation*}
and then applying the result in Remark \ref{rem:non_Hermitian_root} to the
terms on the RHS. Thus, in order to complete the proof, we only need to show
that the ESD of $D_n$ and $E_n$ are almost surely equivalent. To this end, we
need the following proposition.
\begin{pro}\label{proposition_levy_F_D_and_F_E}
Suppose that the data matrix $Y_n=A_{p}^{1/2}X_nB_{n}^{1/2}$, where $A_{p}$ is
defined in (\ref{eq:definition_A_p}), $B_n$ is a $n\times n$ matrix satisfying
$\lim_{n\to \infty}n^{-1}\tr(B_n)=1$. Let $D_n=\sqrt{n/p}(\bs{D}(S_n)-
\bar{b}_n A_p)$ and $E_n=\sqrt{n/p}(\bs{\lambda}(S)- \bar{b}_n A_p)$. Then,
\begin{equation}\label{eq:levy_D_n_E_n}
L(F^{D_n}, F^{E_n})\stackrel{a.s.}{\longrightarrow} 0,
\end{equation}
where $L(F^{D_n},F^{E_n})$ denotes the L\'{e}vy distance between the ESDs of
$D_n$ and $E_n$.
\end{pro}

\begin{proof}
The first step towards proving Proposition \ref{proposition_levy_F_D_and_F_E}
is to obtain a bound on $L(F^{D_n}, F^{E_n})$ in terms of the differences
between eigenvalues of $D_n$ and $E_n$. Observe first that the eigenvalues of
$E_n$ (not necessarily ordered) are given by $\nu_i = \sqrt{n/p}(\lambda_i(S_n)
- \gamma_i)$, where $\gamma_i$ denotes the $i$-th largest eigenvalue of
$\bar{b}_n A_{p}$. This means in particular that
\begin{equation*}
\gamma_{\sum_{k=1}^{j-1} p_k + l} = \bar{b}_n \alpha_j, ~~l=1,\ldots,p_j,
~~j=1,\ldots,m.
\end{equation*}
Next, since $D_n$ is a block diagonal matrix with diagonal blocks
$S_{jj}-\bar{b}_n \alpha_j I_{p_j}$, whose eigenvalues are given by
$\lambda_k(S_{jj}) - \bar{b}_n \alpha_j$, and since $p/n \to 0$ implies that
$\max_{1\leq k \leq p_j}|\lambda_k(S_{jj}) - \bar{b}_n \alpha_j| \to 0$ a.s.,
it follows that for large enough $n$, almost surely, the eigenvalues of $D_n$
are given by $\mu_i = \sqrt{n/p}(\lambda_i(\mathbf{D}(S_n)) - \gamma_i)$, where
$\gamma_i$'s are as defined above. Thus, applying Lemma
\ref{lem:levy_distance_inequality}, we obtain that, for large enough $n$,
almost surely,
\begin{eqnarray}\label{eq:levy_distance_diagonal}
L^2(F^{D_n},F^{E_n}) &\leq& \frac{1}{p} \sum_{i=1}^p |\mu_i - \nu_i|
\nonumber\\
&=& \frac{1}{p} \sum_{i=1}^p  \sqrt{n/p}|\lambda_i(\mathbf{D}(S_n)) -
\lambda_i(S_n)|.
\end{eqnarray}
From (\ref{eq:levy_distance_diagonal}), it is clear that in order to establish
(\ref{eq:levy_D_n_E_n}) it suffices to show that
\begin{equation}\label{eq:max_eigenvalue_difference}
\max_{1\leq i \leq p} \sqrt{\frac{n}{p}} |\lambda_i(\mathbf{D}(S_n)) -
\lambda_i(S_n)| \stackrel{a.s.}{\longrightarrow} 0.
\end{equation}

We prove (\ref{eq:max_eigenvalue_difference}) for $m=2$. The result for general
$m$ follows by a slight modification of the argument and using a finite
induction. In the following, we use the notation $\xi_n = O_{a.s.}(c_n)$ to
mean that $\xi_n/c_n$ is almost surely bounded for large enough $n$. We need
the following well-known result.

\begin{lemma}(Wielandt's Inequality in Eaton and Tyler \cite{ET})\label{lem:wielandt_inequality}
Consider a Hermitian matrix
\begin{eqnarray*}
A=
\begin{pmatrix}
B       & C\\
C^* & D
\end{pmatrix},
\end{eqnarray*}
where $A$ is $p\times p$ and $B$ is $q\times q$ and $D$ is $r\times r$. Let
$\rho^{2}(C)$ denote the largest eigenvalue of $CC^*$ and let $\alpha_1\geq
\cdots\geq \alpha_p$; $\beta_1\geq \cdots\geq \beta_q$ and $\delta_1 \geq
\cdots\geq \delta_r$ denote the ordered eigenvalues of $A$, $B$ and $D$
respectively. If $\beta_q > \delta_1$, then
\begin{equation}\label{eg:wielandt_1}
0 \leq \alpha_j -\beta_j \leq \rho^2(C)/(\beta_j-\delta_1) \qquad j = 1,
\cdots, q
\end{equation}
and
\begin{equation}\label{eg:wielandt_2}
0\leq \delta_{r-i}-\alpha_{p-i}\leq \rho^2(C)/(\beta_q-\delta_{r-i}) \qquad i
=1, \cdots,r-1.
\end{equation}
\end{lemma}

When $m = 2$, we have
\begin{equation*}
S_n =
\begin{pmatrix}
S_{11} & S_{12}\\
S_{21} & S_{22}
\end{pmatrix}
\qquad and \qquad \bs{D}(S_n) =
\begin{pmatrix}
S_{11} & \mathbf{0}\\
\mathbf{0} & S_{22} \\
\end{pmatrix}.
\end{equation*}
Note that since $\alpha_1 > \alpha_2$ and $\parallel S_{jj} - \bar{b}_n
\alpha_j I_{p_j}\parallel \to 0$ a.s., for $j=1,2$, for large enough $n$ we
have, $\lambda_{p_1}(S_{11}) > \lambda_{1}(S_{22})$, almost surely. Thus,
applying Lemma \ref{lem:wielandt_inequality} to
$\lambda_{i}(S_n)-\lambda_{i}(\bs{D}(S_n))$ for $i = 1, 2 \cdots p_1$, we have
\begin{equation}
\lambda_{i}(S_n)-\lambda_{i}(\bs{D}(S_n)) =
\lambda_{i}(S_n)-\lambda_{i}(S_{11})\leq
\frac{\|S_{12}\|^2}{\lambda_{p_1}(S_{11})-\lambda_1(S_{22})}.
\end{equation}
On the other hand, for $i =1, \cdots p_2-1$, we have
\begin{equation}
\lambda_{p-i}(\bs{D}(S_n))-\lambda_{p-i}(S_n)
=\lambda_{p_2-i}(S_{22})-\lambda_{p-i}(S_n) \leq
\frac{\|S_{12}\|^2}{\lambda_{p_1}(S_{11})-\lambda_{1}(S_{22})}.
\end{equation}

Since $\lambda_{p_1}(S_{11})\stackrel{a.s.}{\longrightarrow} \alpha_1$ and
$\lambda_{1}(S_{22})\stackrel{a.s.}{\longrightarrow} \alpha_2$, we have for $j
=1, \cdots p,$
\begin{equation}\label{equation:difference_on_eigenvalues}
\left|\lambda_{j}(S_n)-\lambda_{j}(\bs{D}(S_n))\right|\leq
\left(\frac{1}{\alpha_1-\alpha_2}+O_{a.s}(1)\right)\|S_{12}\|^2
\end{equation}
We will show that
\begin{equation}\label{eq:S_12_square}
\|S_{12}\|^2=O_{a.s.}(p/n),
\end{equation}
which implies (\ref{eq:max_eigenvalue_difference}).

Showing (\ref{eq:S_12_square}) is equivalent to showing that
$\|S_{21}\|=O_{a.s.}(\sqrt{p/n})$. Observe that, we can write $X_{n}^{*}= [X_1^* :
X_2^*]$ where $X_j$ is $p_j \times n$. Also $V_j^* = [V_{j1}^*:V_{j2}^*]$, for
$j=1,2$, where $V_{11}$ is $p_1 \times p_1$, $V_{12}$ is $p_2 \times p_1$,
$V_{21}$ is $p_1\times p_2$ and $V_{22}$ is $p_2 \times p_2$ matrix. Then,
\begin{eqnarray}
S_{21} := \frac{1}{n} Y_2 Y_1^*
&=& \sqrt{\alpha_1\alpha_2} \frac{1}{n} V_2^* X_n B_{n} X_n^* V_1 \nonumber\\
&=& \sqrt{\alpha_1\alpha_2} \frac{1}{n} (V_{21}^* X_1 + V_{22}^* X_2) B_{n} (X_1^* V_{11} + X_2^* V_{12}) \nonumber\\
&=& \sqrt{\alpha_1\alpha_2} \left(\frac{1}{n}  V_{22}^* X_2 B_{n} X_1^* V_{11}
+ \frac{1}{n}  V_{21}^* X_1 B_{n} X_2^* V_{12} \right) \nonumber\\
&& + \sqrt{\alpha_1\alpha_2} \left(\frac{1}{n}  V_{21}^* X_1 B_{n} X_1^* V_{11}
+   \frac{1}{n}  V_{22}^* X_2 B_{n} X_2^* V_{12}\right) \nonumber\\
&=& \sqrt{\alpha_1\alpha_2} \left(\frac{1}{n}  V_{22}^* X_2 B_{n} X_1^* V_{11}
+ \frac{1}{n}  V_{21}^* X_1 B_{n} X_2^* V_{12} \right) \nonumber\\
&& + \sqrt{\alpha_1\alpha_2} V_{21}^* \left(\frac{1}{n}  X_1 B_{n} X_1^* -  \frac{\tr(B_{n})}{n} I_{p_1}\right) V_{11} \nonumber\\
&& + \sqrt{\alpha_1\alpha_2} V_{22}^* \left(\frac{1}{n}  X_2 B_{n} X_2^* -  \frac{\tr(B_{n})}{n} I_{p_2}\right) V_{12},\\
&=&(\uppercase\expandafter{\romannumeral1}+\uppercase\expandafter{\romannumeral2})
+\uppercase\expandafter{\romannumeral3}+\uppercase\expandafter{\romannumeral4}
\end{eqnarray}
where the last step follows from the fact that $\mathbf{0} = V_2^*V_1 =
V_{21}^* V_{11} + V_{22}^* V_{12}$.

We first show that
$\|\uppercase\expandafter{\romannumeral3}\|=O_{a.s.}(\sqrt{p/n})$. Note that by
Lemma 5.3 in Vershynin \cite{Vershynin},
\begin{eqnarray}\label{eq:bound_term_III}
\|\uppercase\expandafter{\romannumeral3}\|&=&\sqrt{\alpha_1\alpha_2}
\|V_{21}^{*}\left(\frac{1}{n}X_{1}B_{n}X_{1}^{*}-\frac{\tr(B_{n})}{n}I_{p_1}\right)V_{11}\|\nonumber\\
&\leq& \sqrt{\alpha_1\alpha_2}\|\frac{1}{n}X_{1}B_{n}X_{1}^{*}-\frac{\tr(B_{n})}{n}I_{p_1}\|\nonumber\\
&=&\sup_{\mbf{a}\in
\mathcal{S}^{p_1-1}}\left|\mbf{a}^{*}\left(\frac{1}{n}X_1B_{n}X_{1}^{*}
-\frac{\tr(B_{n})}{n}I_{p_1}\right)\mbf{a}\right|\nonumber\\
&\leq& (1-\epsilon)^{-1}\max_{\mbf{a}\in
\mathcal{N}_{\epsilon}(\mathcal{S}^{p_1-1})}\left|
\frac{1}{n}\mbf{a}^{*}X_{1}B_{n}X_{1}^{*}\mbf{a}-\frac{\tr(B_{n})}{n}\right|,
\end{eqnarray}
where $\mathcal{N}_{\epsilon}(\mathcal{S}^{p_1-1})$ is an $\epsilon$-net
covering the sphere $\mathcal{S}^{p_1-1}$ with the cardinality
$|\mathcal{N}_{\epsilon}|\leq (1+2/\epsilon)^{p_{1}}$. We need to show that,
for any $\eta>0,~\exists C_{\eta}>0$ such that
\begin{equation}\label{eq:quadratic_form_tail_bound}
\mP\left(\max_{\mbf{a}\in \mathcal{N}_{\epsilon}(\mathcal{S}^{p_1-1})}\left|
\frac{1}{n}\mbf{a}^{*}X_{1}B_{n}X_{1}^{*}\mbf{a}-\frac{\tr(B_{n})}{n}\right|>C_{\eta}
\sqrt{\frac{p}{n}}\right)\leq \exp\{-\eta p\}.
\end{equation}
To this end, we need the following lemma on the concentration of quadratic
forms of sub-Gaussian random variables.
\begin{lemma}\label{lem:Hanson_Wright_inequality}(Hanson-Wright Inequality Theorem 1.1 in Rudelson and Vershynin\cite{RV})
Let $X=(X_1, X_2 ,\cdots, X_n)\in \mR^{n}$ be a random vector with independent
components $X_{i}$ which satisfy $\mE X_{i}=0$ and $\|X_{i}\|_{\psi_2}\leq K$,
where $\parallel \cdot \parallel_{\psi_2}$ denotes the sub-Gaussian norm defined
by $\parallel X_i \parallel_{\psi_2} = \sup_{q \geq 1}
q^{-1/2}(\mathbb{E}|X_i|^q)^{1/q}$. Let $A$ be an $n\times n$ matrix, then for
every $t\geq 0$
\begin{equation}\label{eq:Hanson_Wright_inequality}
\mP\left(|X^{T}AX-\mE X^{T}AX|>t\right)\leq 2 \exp \left\{-c
\mbox{min}\left(\frac{t^2}{K^4\|A\|^2_{HS}}, \frac{t}{K^2\|A\|}\right)\right\}
\end{equation}
\end{lemma}
This lemma applies to both real and complex-valued entries. In order to apply
this result to our setting, we need the vector $X$ to be $y_a = X_1^*
\mathbf{a}$ for any $\mathbf{a}=(a_1,\cdots, a_n)\in S^{n-1}$ and ensure that
there is a uniform finite bound on $\parallel y_{a,i} \parallel_{\psi_2}$ that
does not depend on either $\mathbf{a}$ or $i$. For simplicity, we only provide
the details for the case when $X_{n}$ is real. Thus, let
$y_{a}:=X_{1}^{T}\mathbf{a}$ where $\mathbf{a}=(a_1,\cdots, a_n)\in S^{n-1}$.
Then $y_{a}$ has has i.i.d. sub-Gaussian entries with zero mean, unit variance
and scale parameter $\sigma$. By Lemma 5.5 of Vershynin \cite{Vershynin}, a
random variable is sub-Gaussian if and only if its sub-Gaussian norm is finite
and the sub-Gaussian norm is a constant multiple of the scale parameter
$\sigma$. The sub-Gaussian norm for each entry $y_{a,i}=u_{i}^{T}\mathbf{a}$,
where $u_{i}$ is $i$-th column of $X_{1}$, is given by
\begin{equation*}
\|y_{a,i}\|_{\psi_2}=\sup_{q\geq 1} q^{-1/2} \left(\mE|y_{a,i}|^q\right)^{1/q}.
\end{equation*}
By definition of sub-gaussian random vector and Lemma 5.24 in Vershynin
\cite{Vershynin}, we have for an absolute constant $C$,
\begin{equation*}
\|u_{i}\|_{\psi_{2}}=\sup_{\mathbf{a}\in S^{n-1}}\|u_{i}^T
\mathbf{a}\|_{\psi_{2}}\leq C \max_{1\leq j\leq n}
\|u_{ij}\|_{\psi_2}=C\|u_{11}\|_{\psi_2}:=K.
\end{equation*}
Thus,
\begin{equation*}
\|y_{a,i}\|_{\psi_2}=\|u_{i}^T\mathbf{a}\|\leq \|u_{i}\|_{\psi_{2}}\leq K,
\end{equation*}
and the latter bound does not depend on $\mathbf{a}$ or $i$. Thus, applying
Lemma \ref{lem:Hanson_Wright_inequality}, we can derive that for any $\eta:=
\eta(K, \bar{b}_2)$, there exists $C_{\eta}>0$ and  $N_{\epsilon,\eta}>0$ such
that for $n> N_{\epsilon,\eta}$
\begin{equation*}
\mP\left(\left|\frac{1}{n}y_a^* B_{n}y_a-\frac{1}{n}\tr(B_n)\right|>
C_{\eta(K)}\sqrt{\frac{p}{n}}\right)\leq 2\exp \left\{-\frac{cp}{K^4
\tr(B^2_n)/n}\right\}\leq 2\exp\left\{-\eta p\right\}.
\end{equation*}
This proves (\ref{eq:quadratic_form_tail_bound}).
Thus, by Borel-Cantelli lemma and (\ref{eq:bound_term_III}), we have
$\|\uppercase\expandafter{\romannumeral3}\|=O_{a.s.}(\sqrt{p/n})$. Similarly,
$\|\uppercase\expandafter{\romannumeral4}\|=O_{a.s.}(\sqrt{p/n})$.

Next, we show that
$\|\uppercase\expandafter{\romannumeral1}\|=O_{a.s.}(\sqrt{p/n})$. Note that,
since $V_j^* V_j = I_{p_j}$ for $j=1,2$, we have $\parallel V_{jk}\parallel
\leq 1$ for $1\leq j,k \leq 2$, and hence
\begin{eqnarray*}
\|\uppercase\expandafter{\romannumeral1}\|=\sqrt{\alpha_1\alpha_2}\|V_{22}^{*}X_2B_{n}X_{1}^{*}V_{11}\|\leq
\sqrt{\alpha_1\alpha_2}\|\frac{1}{n}X_1B_nX_2^{*}\|.
\end{eqnarray*}
Let $U_{12}=\frac{1}{n}X_1B_nX_2^{*}$.
We will prove that
$\|U_{12}U_{12}^*\|=\|\frac{1}{n}X_1B_n\frac{1}{n}(X_2^{*}X_{2})B_{n}X_{1}^{*}\|=
O_{a.s.}(p/n)$. Accordingly, define
$\widetilde{D}=n^{-1}B_{n}X_{2}^{*}X_{2}B_{n}$, and note that $\widetilde{D}$
has the same $p_{2}$ non-zero eigenvalues as $n^{-1}X_{2}B_{n}^{2}X_{2}^*$ as
well as $(n-p_2)$ zero eigenvalues. Let
$\widetilde{D}:=Q\widetilde{\Lambda}Q^{*}$ denote the spectral decomposition of
$\widetilde{D}$ where $Q$ is $n \times p_2$ and $\widetilde{\Lambda}$ is a $p_2
\times p_2$ diagonal matrix. Define $\widetilde{X}_{1}=X_{1}Q$ and observe that
the rows of $\widetilde{X}_{1}$ are i.i.d. and sub-Gaussian conditionally on
$X_2$. Also note that if $x_j^*$ and $\widetilde x_j^*$ denote the $j$-th row
of $X_1$ and $\widetilde X_1$, respectively, then $\widetilde x_j = Q^* x_j$,
so that $\mE(\widetilde{x}_{j}\widetilde{x}_{j}^{*}|X_2) =Q^*
\mE(x_{j}x_{j}^{*})Q = I_{p_2}$, which shows that rows of $\widetilde{X}_{1}$
are isotropic random vectors conditionally on $X_2$. In addition,
\begin{equation*}
\parallel U_{12} U_{12}^* \parallel = \parallel \frac{1}{n}\widetilde{X}_{1}\widetilde{\Lambda}\widetilde{X}_{1}^{*} \parallel
\leq \frac{p_2}{n} \parallel \frac{1}{p_2} \widetilde{X}_{1}
\widetilde{X}_{1}^{*}\parallel \parallel \widetilde\Lambda \parallel.
\end{equation*}
Then by Lemma 5.9 and Theorem 5.39 in Vershynin \cite{Vershynin}, applied to the matrix
$p_2^{-1/2} \widetilde{X}_1$, and the fact that the entries of the diagonal
matrix $\widetilde{\Lambda}$ are bounded by $\parallel B_n\parallel^2\parallel
n^{-1} X_2 X_2\parallel$ which is a.s. finite (again, by Theorem 5.39 in
Vershynin \cite{Vershynin}), from the above display we conclude that $\parallel U_{12}
U_{12}^* \parallel = O_{a.s.}(p/n)$. Hence,
$\|\uppercase\expandafter{\romannumeral1}\|= O_{a.s.}(\sqrt{p/n})$ and
similarly,  $\|\uppercase\expandafter{\romannumeral2}\|= O_{a.s.}(\sqrt{p/n})$.
So we finish the proof of Proposition \ref{proposition_levy_F_D_and_F_E} for
$m=2$ case.

We now give a brief outline of the induction argument. Suppose that $\alpha_1
> \cdots > \alpha_m \geq 0$ and (\ref{eq:max_eigenvalue_difference}) holds for
$m=M-1$. We want to establish that the same holds when $m=M$. Accordingly, we
write $S_n$ as
$$
S_n = \begin{pmatrix}
\widetilde S_{M-1,M-1} & \widetilde S_{M-1,M}\\
\widetilde S_{M-1,M}^*& S_{MM}\\
\end{pmatrix}
~~\mbox{and define}~~ \widetilde{\mathbf{D}}(S_n) = \begin{pmatrix}
\widetilde S_{M-1,M-1} & \mathbf{0}\\
\mathbf{0} & S_{MM}\\
\end{pmatrix},
$$
where $\widetilde S_{M-1,M-1}$ is the $(p-p_M)\times (p-p_M)$ principal
submatrix of $S_n$, $\widetilde S_{M-1,M}$ is $(p-p_M) \times p_M$ and $S_{MM}$
is $p_M \times p_M$. The proof follows by first showing that
$$
\max_{1\leq i \leq p} \sqrt{\frac{n}{p}}|\lambda_i(\widetilde{\mathbf{D}}(S_n))
- \lambda_i(S_n)| \stackrel{a.s.}{\longrightarrow} 0,
$$
through proving that $\parallel \widetilde S_{M-1,M}\parallel^2 =
O_{a.s.}(p/n)$, which requires a minor modification of the argument for showing
(\ref{eq:S_12_square}), and then applying the induction hypothesis. The details
are omitted.
\end{proof}

\section{Simulation}\label{sec:simulation}
In this section we carry out a simulation study to:
\begin{itemize}
\item[(i)] demonstrate  the
convergence of the ESD of $C_n$ to the limiting distribution by considering different
combinations of $(p,n)$;
\item[(ii)] to illustrate the performance of the test based on the statistic $L_n$ proposed in
Section \ref{subsec:application} by considering a specific null $H_{0}:A=A_0;
B=B_{0}$ versus a specific simple alternative $H_{1}: A=A_1;B=B_{1}$.
\end{itemize}
We numerically investigate the convergence of the ESD of
$C_n = \sqrt{n/p}(n^{-1}Y_nY_n^* - n^{-1}\tr(B_{0}) A_{0})$
to the LSD under $H_0$, viz. $F^{A_0,B_0}$. Note that, under $H_0$, $Y_n = A_{0}^{1/2} X_n B_{0}^{1/2}$,
while under $H_1$, $Y_n = A_{1}^{1/2} X_n B_{1}^{1/2}$.
We specifically assume that
\begin{equation}\label{eq:A_0_B_0}
A_{0}=\mbox{diag}(\underbrace{1\cdots 1}_{p/3},
\underbrace{2 \cdots 2}_{p/3}, \underbrace{3 \cdots 3}_{p/3}),~~B_{0}= I_{n},
\end{equation}
and
\begin{equation}\label{eq:A_1_B_1}
A_{1}=\mbox{diag}(\underbrace{1.5\cdots 1.5}_{p/3},
\underbrace{2 \cdots 2}_{p/3}, \underbrace{2.5 \cdots 2.5}_{p/3}) ;~B_{1}=
I_{n}.
\end{equation}
Since $B_{n}$ only influences the scale of the spectrum through the factor
$n^{-1}\tr(B_{n}^2)$, for ease of comparison we take $B_{0} = B_{1} = I_n$.

First, to empirically investigate the rate of convergence of the ESD to the LSD,
we simulate data under $H_0$ and plot the relative frequency histogram of eigenvalues of $C_n$
together with the density of the LSD $F^{A_0,B_0}$, denoted by
$f_0$. As indicated in Section \ref{subsec:computation_density}, this involves
solving the following equation for $\beta(x)$:
\begin{equation}\label{polynomial_equation}
18\beta(x)^{4}+33x\beta(x)^{3}+(18x^2+18)\beta(x)^{2}+(3x^{3}+22x)\beta(x)+6x+2
= 0.
\end{equation}
The histograms for five different combinations of $(p,n)$ are shown in Figure \ref{fig:multi_compare}.
As we can see, with increasing values of $p$ and $n$ such that $p/n$ becomes smaller, the
histograms closely match the smooth curve representing the density $f_0$ of the LSD.


In addition to the graphical comparison, we also compute the value of the statistic $L_n$ defined
in (\ref{eq:CVM_statistic}), which measures the discrepancy between the ESD of $C_n$ (when
the data follow $H_0$) and the LSD $F^{A_0,B_0}$. We make a three-way comparison, namely,
(i) fixing $p$ and letting $n$ increase; (ii) fixing $p$ and letting $n$ increase; and
(iii) allowing both $p$ and $n$ increase such that $p/n \to 0$. The third
scenario connects directly to the theory developed in this paper. The values of the means
and standard deviations of the statistic $L_n$ based on 100 replicates for each
of the $(p,n)$ combinations are reported in Table \ref{table:convergence_rate}.
\begin{itemize}
\item \textit{Fix $p$, increase $n$:} Along the rows of Table
\ref{table:convergence_rate}, i.e., for a fixed $p$, as $n\to \infty$, the matrix
$C_{n}$ converges in distribution to a
matrix of the form $G_p = p^{-1/2} \sqrt{\bar{b}_2} A_p^{1/2} W_p A_p^{1/2}$ where $W_p$ is a $p\times p$
(real or complex) Wigner matrix, and so the ESD of $C_n$ converges to that of $G_p$
which is different from $F^{A_0,B_0}$. As can be seen from Table \ref{table:convergence_rate}, that
along the rows, with increasing $n$, the mean value of $L_{n}$ stabilizes to a nonzero
value due to the fact that
the LSD of $F^{C_n}$ is a limit distribution that is different from $F^{A_0,B_0}$.
\item
\textit{Fix $n$, increase $p$:} This comparison relates to the columns of Table
\ref{table:convergence_rate}.  The limiting behavior of $F^{C_n}$ under this setting
is unclear and is beyond the scope of this paper. However, for any given $n$,
for large enough $p$, the ESD of $C_n$ will be quite different
from $F^{A_0,B_0}$.
\item\textit{$p$, $n$ both increase such that $p/n\to 0$:} This is the setting
studied in this paper. For this comparison, we focus on the main diagonal of Table
\ref{table:convergence_rate}
Under this setting, $F^{C_n}$ converges to $F^{A_0,B_0}$ almost surely. The
mean and $\pm 2$ standard deviation bars are depicted in Figure \ref{figure:ratio_mean_std_L},
with $p/n$ taking values $33/1000$, $66/3300$, $99/10000$, $201/40000$ and $600/240000$,
respectively. We observe that both the mean and standard deviation of $L_n$ decrease to zero
as $p/n$ decreases to zero. This observation is consistent with
the comparison of the histograms of eigenvalues of $C_n$ for the same combinations
of $(p,n)$ as depicted in Figure \ref{fig:multi_compare}.
\end{itemize}
\begin{table}[H]
\begin{center}
\caption{Mean and standard deviations (within parentheses) of $L_n$ under $H_0$.} {\footnotesize}
\begin{tabular}{|c |c*{4}{c}|c|}
\hline
~~~~~~$n$  &1000&3300&10000&40000& 240000\\
$p$ &&&& &\\
\hline
33& 0.0050 &0.0044 &0.0042 &0.0037 & 0.0041\\
& (0.0021) & (0.0020) & (0.0018) & (0.0015) & (0.0017)\\
\hline
66 &0.0033&0.0018&0.0013&0.0011&0.0011\\
      &(8.9903e-4)&(6.1469e-4)&(4.5269e-4)&(3.4770e-4)&(3.7956e-4)\\
\hline
99 &0.0037& 0.0015&8.3441e-4&6.5708e-4&5.6750e-4\\
      &(8.0365e-4)&(4.2526e-4)&(2.2154e-4)&(2.6689e-4)&(2.0820e-4)\\
\hline
201 &0.0065&0.0020&8.1588e-4&3.0589e-4&1.7812e-4 \\
        &(5.0315e-4)&(2.7464e-4)&(1.7132e-4)&(8.0617e-5)&(6.2289e-5)\\
\hline
600 &0.0193&0.0058&0.0019&4.9400e-4&1.0062e-4\\
        &(2.7617e-4)&(1.4565e-4)&(8.4915e-5)&(3.9138e-5)&(1.7237e-5)\\
\hline
\end{tabular}
\label{table:convergence_rate}
\end{center}
\end{table}
\begin{figure}
\centering
\includegraphics[width=0.75\textwidth]{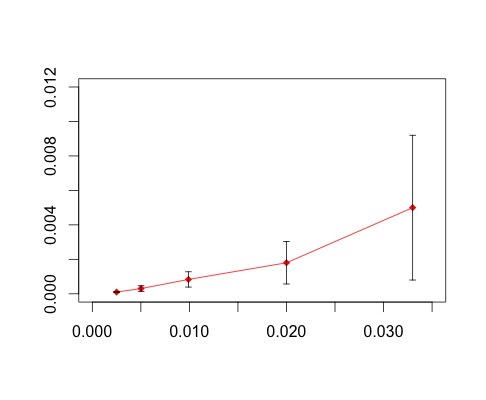}
\caption{Mean $\pm$ 2 $\times$ standard deviation of $L_n$ under different $p/n$ ratios.}
\label{figure:ratio_mean_std_L}
\end{figure}

\begin{figure}
\centering
\includegraphics[width=0.75\textwidth]{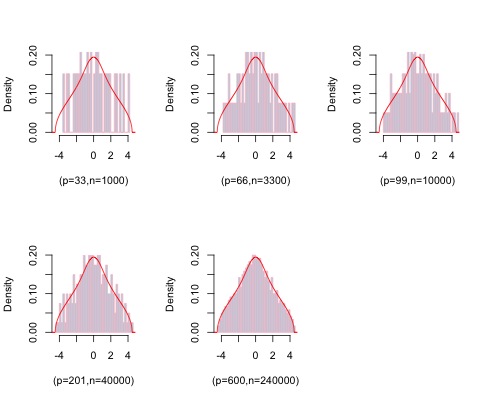}
\caption{Histogram of the eigenvalues under five different combinations of $(p,n)$.}
\label{fig:multi_compare}
\end{figure}

Next, we show the performance of the test for $H_{0}:A=A_0;
B=B_{0}$ versus $H_{1}: A=A_1;B=B_{1}$ based on the test statistic of $L_n$, where $A_{j}$ and $B_{j}$, $j=0,1$
are defined in (\ref{eq:A_0_B_0}) and (\ref{eq:A_1_B_1}).
Rather than performing the test at a specific level of significance, we compute
the quantiles of the distribution of $L_n$ under $H_0$ and $H_1$ corresponding
to a given set of probabilities. In order to
evaluate the quantiles empirically, we simulate 500 replicates for each setting.
The quantiles of the test statistics $L_n$ under $H_1$ are plotted against
those under $H_0$ in Figure \ref{l2_qq_plot}. Since the points lie well above
the $45^o$ line, it shows that the test is able to reject
the null hypothesis at any reasonable level of significance when the data
are generated under the alternative.

\begin{figure}[t]
\centering
\includegraphics[width=0.75\textwidth]{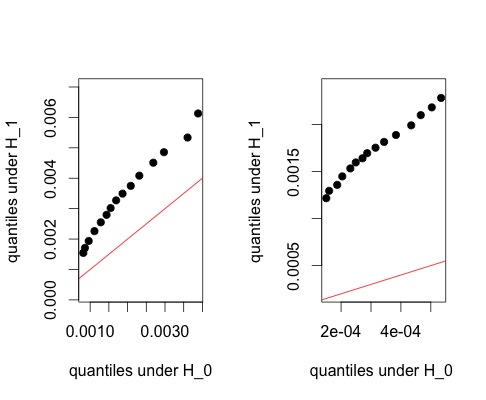}
\caption{QQ plot of the test statistic $L_n$ under $H_{0}$ versus under
$H_{1}$. Left panel: $p=66, n=3300$; right panel : $p=201, n=40000$.}
\label{l2_qq_plot}
\end{figure}

The numerical values of the quantiles of the distribution of $L_n$ under $H_0$
and $H_1$ are given in Table \ref{table:quantiles}. It shows that especially
for $p = 201$; $n = 40000$ setting, the effective supports of the distributions
of the test statistic are essentially separated under $H_0$ and $H_1$, indicating that the test
is able to clearly discriminate between the two hypotheses.

\begin{table}[H]
\begin{center}
\caption{Quantiles of $L_n$ under $H_0$ and $H_1$ for $(p,n) = (66,3300)$ and $(201,40000)$.} {\footnotesize}
\begin{tabular}{|c|cc|cc|}
\hline
 &  \multicolumn{2}{c|}{(66,3300)} & \multicolumn{2}{c|}{(201,40000)}\\
Probability      &$H_0$&$H_1$&$H_0$&$H_1$\\
 \hline
0.01              &0.0008&0.0015&1.506e-4& 0.0012\\
0.02              &0.0009&0.0017&1.601e-4&0.0013\\
0.05              &0.0010&0.0019&1.871e-4&0.0014\\
0.1                &0.0011&0.0023&2.040e-4&0.0014\\
0.2                &0.0013&0.0026&2.309e-4&0.0015\\
0.3                &0.0014&0.0028&2.491e-4&0.0016\\
0.4                &0.0015&0.0030&2.716e-4&0.0016\\
0.5                &0.0017&0.0033&2.872e-4&0.0017\\
0.6                &0.0019&0.0035&3.151e-4&0.0018\\
0.7                &0.0021&0.0037&3.438e-4&0.0018\\
0.8                &0.0023&0.0041&3.835e-4&0.0019\\
0.9                &0.0027&0.0045&4.345e-4&0.0020\\
0.95              &0.0030&0.0049&4.664e-4&0.0021\\
0.98              &0.0036&0.0053&5.033e-4&0.0022\\
0.99              &0.0039&0.0061&5.343e-4&0.0023\\
\hline
\end{tabular}
\label{table:quantiles}
\end{center}
\end{table}

\section*{\ackname}
The authors thank the anonymous referees for their valuable suggestions regarding
improving the quality of the manuscript.
This work was done during a visit of the first author to the Department of
Statistics, University of California, Davis. Wang was partially supported by
NSFC grant 11071213, NSFC 11371317, NSFC grant 11101362, ZJNSF grant R6090034 and SRFDP grant
20100101110001. Paul was partially supported by the NSF grants DMR-1035468 and
DMS-1106690.


\newpage

\section{Appendix}

\subsection{Auxiliary lemmas}

\begin{lemma}\label{lemma:quad_bound}
(Lemma 2.6 of Silverstein and Bai \cite{SB}): Let $z \in \mathbb{C}^+$ with $v
= \Im(z)$. Let $D$ and $F$ be $n \times n$ matrices with $D$ Hermitian, and let
$\mbf{r} \in \mathbb{C}^n$. Then,
$$
\left|\tr\left(((D-zI)^{-1} - (D+\mbf{r}\mbf{r}^* - zI)^{-1})F\right)\right| =
\left|\frac{\mbf{r}^* (D-zI)^{-1} F (D-zI)^{-1}
\mbf{r}}{1+\mbf{r}^*(D-zI)^{-1}\mbf{r}}\right| \leq \frac{\parallel F
\parallel}{v}~.
$$
\end{lemma}

\begin{lemma}\label{eg:burkholder_inequality}
(Burkh\"{o}lder's Inequality): Let $\{X_{k}\}$ be a complex martingale
difference sequence with respect to the increasing $\sigma$-field
$\{\mathcal{F}_{k}\}$. Then for $p >1,$
$$\mE\left|\sum X_{k}\right|^{p}\leq K_{p}\mE\left(\sum |X_{k}|^2\right)^{p/2}.
$$
\end{lemma}

\begin{lemma}\label{lemma:moments_of_quadratic_forms}
(Lemma 8.10 of Silverstein and Bai \cite{SB}): Let $A = (a_{ij})$ be an
$n\times n$ non-random matrix and $X=(x_1, \cdots, x_n)^{'}$ be random vector
of independent entries. Assume that $\mE x_{i}=0,$ $\mE|x_{i}|^2=1$ and
$\mE|x_{j}|^{l}\leq \nu_{l}$. Then for any $q\geq 2$,
\begin{equation*}
\mE|X^{*}AX-\tr(A)|^{q}\leq
C_{q/2}\left(\nu_{2q}\tr(AA^{*})^{q/2}+(\nu_{4}tr(AA^*)^{q/2})\right)
\end{equation*}
where $C_{q}$ is a constant depending on $p$ only.
\end{lemma}

The following lemma is a consequence of Theorem A.38 and Remark A.39 in Bai and
Silverstein \cite{BSbook}.
\begin{lemma}\label{lem:levy_distance_inequality}
Let $\{\lambda_k\}_{k=1}^n$ and $\{\delta_k\}_{k=1}^n$ be two sets of real and
let their empirical distributions be denoted by $F$ and $G$, respectively.
Then, for any $\alpha > 0$,
\begin{equation}\label{eq:levy_distance_inequality}
L^{\alpha + 1}(F,G)\leq \min_{\pi} \frac{1}{n} \sum_{k=1}^n |\lambda_k -
\delta_{\pi(k)}|^\alpha,
\end{equation}
where the minimum is taken over all permutation $\pi$ of the indices
$\{1,\ldots,n\}$, and $L(F,G)$ denotes the L\'{e}vy distance between the
distributions $F$ and $G$.
\end{lemma}

\begin{lemma}\label{lem:vershynin_berstein_tail}
(Bernstein's inequality): Let $X_{1}, \cdots, X_{N}$ be independent centered
sub-exponential random variables, and $K=\max_{i}\|X_{i}\|_{\psi_{1}}$ where
$\|X\|_{\psi_1}=\sup_{p\geq 1}p^{-1}\left(\mE|X|^p\right)^{1/p}.$ Then for
every $a=(a_1, \cdots, a_{N})\in \mR^{N}$ and every $t\geq 0,$ we have
\begin{equation*}
\mP\left(\left|\sum_{i=1}^{N}a_{i}X_{i}\right|>t\right)\leq
2\exp\left\{-c\min\left(\frac{t^2}{K^2\|a\|_{2}^2},
\frac{t}{K\|a\|_{\infty}}\right)\right\}
\end{equation*}
\end{lemma}

\begin{lemma}\label{lem:vershynin_subexponential_tail}
(Corollary 5.17 in Vershynin \cite{Vershynin}): Let $X_{1}, \cdots, X_{N}$ be
independent centered sub-exponential random variables, and let
$K=\max_{i}\|X_{i}\|_{\psi_1}$ where $\|X\|_{\psi_1}=\sup_{p\geq
1}p^{-1}\left(\mE|X|^p\right)^{1/p}.$ Then for every $\epsilon\geq 0,$ we have
\begin{equation*}
\mP\left(\left|\sum_{i=1}^{N}X_{i}\right|\geq \epsilon N\right)\leq
2\exp\left\{-c\min\left(\frac{\epsilon^2}{K^2},
\frac{\epsilon}{K}\right)N\right\}
\end{equation*}
where $c>0$ is an absolute constant.
\end{lemma}

\begin{lemma}\label{lem:hoeffding_type_inequality}
(Hoeffding's inequality : Proposition 5.10 in Vershynin \cite{Vershynin}): Let
$X_{1}, \cdots, X_{N}$ be independent centered sub-gaussian random variables,
and let $K=\max_{i}\|X_{i}\|_{\psi_{2}}$ where $\|X\|_{\psi_2}=\sup_{p\geq
1}p^{-1/2}\left(\mE|X|^p\right)^{1/p}.$ THen for every $a=(a_{1}, \cdots,
a_{N})\in \mR^{N}$ and every $t\geq 0,$ we have
\begin{equation*}
\mP\left(\left|\sum_{i=1}^{N}a_{i}X_{i}\right|\geq t\right)\leq e\cdot
\exp\left\{-\frac{ct^2}{K^2\|a\|_{2}^{2}}\right\}
\end{equation*}
where $c>0$ is an absolute constant.
\end{lemma}

\subsection{Bound on $d_{1}$ : proof of (\ref{eq:Y_eh_Y_lambda_bound})}\label{attach_estimation}

This is a direct application of the strategy shown in Section
\ref{subsec:random part}. We will show that
$d_1=\frac{1}{p}\mE\left|tr\left[\left(Y^{-1}(z)-Y_{(k)}^{-1}(z)\right)\Lambda\right]\right|\leq
M/p$. To this end, we repeat the computation in
(\ref{eg:difference_resolvent_deleting_kth_row}). Since
\begin{equation}\label{eg:Y_z_decomposition}
Y(z)=\sqrt{\frac{n}{p}}(VV^*-\bar{b}_n\Lambda)=Y_{(k)}(z)+\omega_{k}e_{k}^{*}+e_{k}\omega_{k}+\tau_{kk}e_{k}e_{k}^{*}.
\end{equation}
Let
$\omega_{k}e_{k}^{*}+e_{k}\omega_{k}=\mbf{u}_{k}\mbf{u}_{k}^{*}-\mbf{v}_{k}\mbf{v}_{k}^{*}$,
where $\mbf{u}_{k}=2^{-1/2}(e_k+\omega_{k})$ and
$\mbf{v}_{k}=2^{-1/2}(e_{k}-\omega_{k})$. Also, define
$D_{1k}=Y_{(k)}+\mbf{u}_{k}\mbf{u}_{k}^{*}$ and
$D_{2k}=D_{1k}-\mbf{v}_{k}\mbf{v}_{k}^{*}$ so that
$D_{1k}=D_{2k}+\mbf{v}_{k}\mbf{v}_{k}^{*}$. Then  from
(\ref{eg:Y_z_decomposition}) we have $Y(z)=D_{2k}+\tau_{kk}e_{k}e_{k}^{*}$.
Therefore,
\begin{eqnarray*}
&&\tr\left[\left(Y^{-1}(z)-Y_{(k)}^{-1}(z)\right)\Lambda\right]\\
&=& \tr\left[\left(Y^{-1}(z)-(D_{2k}-zI)^{-1}\right)\Lambda\right]+ \tr\left[\left((D_{2k}-zI)^{-1}-(D_{1k}-zI)^{-1}\right)\Lambda\right]\\
&& ~~~~+ \tr\left[\left((D_{1k}-zI)^{-1}-Y_{(k)}^{-1}(z)\right)\Lambda\right]\\
&=&
\frac{\tau_{kk}e_{k}^{*}(D_{2k}-zI)^{-1}\Lambda(D_{2k}-zI)^{-1}e_{k}}{1+\tau_{kk}e_{k}^{*}(D_{2k}-zI)^{-1}e_{k}}
+
\frac{\mbf{v}_{k}^{*}(D_{1k}-zI)^{-1}\Lambda(D_{2k}-zI)^{-1}\mbf{v}_{k}}{1+\mbf{v}_{k}^{*}(D_{1k}-zI)^{-1}\mbf{v}_{k}}
+ \frac{\mbf{u}_{k}^{*}Y_{(k)}^{-1}\Lambda
Y_{(k)}^{-1}\mbf{u}_{k}}{1+\mbf{v}_{k}^{*}Y_{(k)}^{-1}(z)\mbf{u}_{k}}.
\end{eqnarray*}
According to (\ref{eq:C_diff}) and Lemma \ref{lemma:quad_bound}, each term
above is bounded by $a_{0}/v$. Thus
\begin{equation}
\frac{1}{p}\mE\left|\tr\left[\left(Y^{-1}(z)-Y^{-1}_{(k)}(z)\right)\Lambda\right]\right|\leq
\frac{3a_{0}}{pv}\leq \frac{M}{p}.
\end{equation}
So we have $d_{1}\leq M/p$.


\subsection{Bound on $d_{31}$}\label{d_31}

Since
\begin{eqnarray*}
d_{31}&=&\mE\left|-\omega_{k}^{*}Y_{(k)}^{-1}(z)\omega_{k}+\tau_{kk}
+\bar{b}_{2}(n)\frac{\lambda_{k}}{p} \tr(Y^{-1}_{(k)}(z)\Lambda_{(k)})\right|^2\\
&\leq& 2\mE
\left|\tau_{kk}\right|^2+2\mE\left|\omega_{k}^{*}Y_{(k)}^{-1}(z)\omega_{k} -
\bar{b}_{2}(n)\frac{\lambda_{k}}{p}\tr\left[Y_{(k)}^{-1}(z)\Lambda_{(k)}\right]\right|^2,
\end{eqnarray*}
and we already have
\begin{equation}\label{eq:E_tau_kk_bound}
\mE |\tau_{kk}|^2\leq \frac{n}{p}\mE
\left|\frac{\lambda_{k}}{n}\widetilde{x}_{k}^{*}B_{n}\widetilde{x}_{k}-\bar{b}(n)\lambda_{k}\right|^{2}\leq
\frac{M}{p},
\end{equation}
to prove the claim that $d_{31}\leq M/p$, we need a bound on the expected value
of the term
\begin{equation*}
\omega_{k}^{*}Y_{(k)}^{-1}(z)\omega_{k}-\bar{b}_{2}(n)\frac{\lambda_{k}}{p}\tr\left[Y_{(k)}^{-1}(z)\Lambda_{(k)}\right]:=d_{k}^{(2)}
\end{equation*}
defined in (\ref{eq:quadratic term}). Note that
\begin{eqnarray*}
d_k^{(2)}&=& \frac{n}{p}v_{k}^{*}V_{(k)}^{*}Y_{(k)}^{-1}(z)V_{(k)}v_{k}
- \bar{b}_{2}(n)\frac{\lambda_{k}}{p}\tr\left[Y_{(k)}^{-1}(z)\Lambda_{(k)}\right]\\
&=&\frac{\lambda_{k}}{p}tr\left[V_{(k)}B_{n}V_{(k)}^{*}Y_{(k)}^{-1}(z)\right]
- \bar{b}_{2}(n)\frac{\lambda_{k}}{p}\tr\left[Y_{(k)}^{-1}(z)\Lambda_{(k)}\right] + d_{k}^{(1)}\\
&=& \frac{\lambda_{k}}{pn}\sum_{i,j \ne k}(\sqrt{\lambda_{i}\lambda_{j}}\widetilde{x}_{i}^{*}B_{n}^{2}\widetilde{x}_{j})
(Y_{(k)}^{-1}(z))_{ji}-\bar{b}_{2}(n)\frac{\lambda_{k}}{p}\tr\left[Y_{(k)}^{-1}(z)\Lambda_{(k)}\right]+d_{k}^{(1)}\\
&=&\frac{\lambda_{k}}{p}\sum_{i\ne k} \lambda_{i}\left(\frac{1}{n}\widetilde{x}_{i}^{*}B_{n}^{2}\widetilde{x}_{i}
-\bar{b}_{2}(n)\right)(Y_{(k)}^{-1}(z))_{ii}+\frac{\lambda_{k}}{p}\sum_{i\ne j\ne k}\sqrt{\lambda_{i}\lambda_{j}}
\frac{1}{n}\widetilde{x}_{i}^{*}B_{n}^{2}\widetilde{x}_{j}(Y_{(k)}^{-1}(z))_{ji}+d_{k}^{(1)}\\
&:=& d_{k}^{(3)}+d_{k}^{(4)}+d_{k}^{(1)}.
\end{eqnarray*}
In order to show $\mE |d_{k}^{(2)}|^2\leq M/p$, we need to derive corresponding
bounds on $\mE|d_{k}^{(3)}+d_{k}^{(4)}|^2$ and $\mE|d_{k}^{(1)}|^2$. Using
Lemma \ref{lemma:moments_of_quadratic_forms}, we have that for any $q \geq 2$
\begin{eqnarray}\label{eg:bound_expectation_quadra}
\mE \left|\frac{1}{n}\widetilde{x}_{i}^{*}B_{n}^{2}\widetilde{x}_{i}
-\bar{b}_{2}(n)\right|^{q}&\leq& C_{q} \left(3^{q/2}(n^{-2}\tr(B_{n}^4))^{q/2}
+ \nu_{2q}n^{-q}\sum_{i=1}^{n}\lambda_{i}(B_n)^{2q}\right)\nonumber\\
&\leq& \frac{C_{q}^{'}}{n^{q}}\left((tr(B_n^{4}))^{q/2}+\tr(B_{n}^{2q})\right)\nonumber\\
&\leq& \frac{C_{q}^{'}}{n^{q}}\left(b_{0}^{q}(n\bar{b}_{2}(n))^{q/2}+b_{0}^{2q-2}(n\bar{b}_2(n))\right)\nonumber\\
&\leq&
C_{q}^{'}\left(\frac{b_{0}^{q}(\bar{b}_{2}(n))^{q/2}}{n^{q/2}}+\frac{b_{0}^{2q-2}\bar{b}_2(n)}{n^{q-1}}\right)\leq
\frac{M}{n^{q/2}}.
\end{eqnarray}
Thus, taking $q=2$ in (\ref{eg:bound_expectation_quadra}) and using
Cauchy-Schwarz inequality, we have
\begin{eqnarray}
\mE \left|d_{k}^{(3)}+d_{k}^{(4)}\right|^2
&=& \mE \left|\frac{\lambda_{k}}{p}\tr\left[\left(V_{(k)}B_{n}V_{(k)}^{*}-\bar{b}_{2}(n)\Lambda_{(k)}\right)
Y_{(k)}^{-1}(z)\right]\right|^2\nonumber\\
&=&\frac{\lambda_{k}^2}{p^2}\mE \left|\tr\left[\left(\frac{1}{n}\widetilde{X}_{(k)}B_{n}^{2}\widetilde{X}_{(k)}^{*}
-\bar{b}_2(n)I_{(k)}\right)\left(\Lambda_{(k)}^{1/2}Y_{(k)}^{-1}(z)\Lambda_{(k)}^{1/2}\right)\right]\right|^2\nonumber\\
&\leq& \frac{a_{0}^2}{p^2} \mE \left[\tr\left(\frac{1}{n}\widetilde{X}_{(k)}B_{n}^{2}\widetilde{X}_{(k)}^{*}
-\bar{b}_2(n)I_{(k)}\right)^2\tr\left(\Lambda_{(k)}^{1/2}Y_{(k)}^{-1}\Lambda_{(k)}Y_{(k)}^{-1}(\bar{z})\Lambda_{(k)}^{1/2}\right)\right]\nonumber\\
&\leq& \frac{a_{0}^4}{pv^2}\mE
\tr\left(\frac{1}{n}\widetilde{X}_{(k)}B_{n}^{2}\widetilde{X}_{(k)}^{*}-\bar{b}_2(n)\Lambda_{(k)}\right)^2,
\end{eqnarray}
where $I_{(k)}=I-e_{k}e_{k}^{*}$. Indeed,
\begin{eqnarray}
\mE\tr\left(\frac{1}{n}\widetilde{X}_{(k)}B_{n}^{2}\widetilde{X}_{(k)}^{*}
-\bar{b}_2(n)\Lambda_{(k)}\right)^2 &=&\sum_{i\ne k} \mE \left(\frac{1}{n}\widetilde{x}_{i}^{*}B_{n}\widetilde{x}_{i}
-\bar{b}_{2}(n)\right)^2+\sum_{i \ne j\ne k}\mE \left(\frac{1}{n}\widetilde{x}_{i}^{*}B_{n}\widetilde{x}_{j}\right)^2\nonumber\\
&\leq& (p-1)\frac{M_{1}}{n}+(p-1)(p-2)\frac{M_2}{n}.
\end{eqnarray}
Then we have
\begin{equation}\label{eg:bound_on_d_3_plus_d_4}
\mE\left|d_{k}^{(3)}+d_{k}^{(4)}\right|^2\leq
\frac{a_{0}^2}{pv^2}\left((p-1)\frac{M_{1}}{n}+(p-1)(p-2)\frac{M_2}{n}\right)\leq
\frac{pM}{n}
\end{equation}
which goes to zero as $n\to \infty$. Next, we show that $\mE
|d_{k}^{(1)}|^2\leq M/p$. Since
\begin{eqnarray}
d_{k}^{(1)}&=&\frac{n}{p}v_{k}^{*}V_{(k)}^{*}Y_{(k)}^{-1}(z)V_{(k)}v_{k}
-\frac{\lambda_{k}}{p}tr\left[V_{(k)}B_{n}V_{(k)}^{*}Y_{(k)}^{-1}(z)\right]\nonumber\\
&=&
\frac{\lambda_{k}}{n}\widetilde{x}_{k}^{*}Q_n(z)\widetilde{x}_{k}-\tr(Q_n(z))
\end{eqnarray}
where $Q_n(z):=B_{n}^{1/2}V_{(k)}^{*}Y_{(k)}^{-1}(z)V_{(k)}B_{n}^{1/2}$, we get
\begin{eqnarray}\label{eg:bound_on_d_1}
\mE|d_{k}^{(1)}|^2
&=& \frac{\lambda_{k}^2}{p^2}\mE
\left[\mE\left(\left|(\widetilde{x}_{k}^{*}Q_n(z)\widetilde{x}_{k}
-\tr(Q_n(z)\right|^2\big| \widetilde{X}_{(k)}\right)\right]\nonumber\\
&\leq& C\frac{a_0^{2}}{p^2} \mE\tr\left(Q_{n}(z)Q_{n}(z)^{*}\right)\nonumber\\
&\leq& \frac{C^{'}}{p}\mE \|Q_n(z)\|^2\nonumber\\
&\leq& \frac{M}{p}.
\end{eqnarray}
The last inequality holds due to the fact that under Gaussianity, we have
\begin{eqnarray*}
\|Q_{n}(z)\|^2 \leq
\|B_n\|^2\|n^{-1}\widetilde{X}_{(k)}\Lambda\widetilde{X}_{(k)}\|\|Y_{(k)}^{-1}(z)\|\leq
\frac{b_{0}^2a_{0}}{v}\|n^{-1}\widetilde{X}_{(k)}\widetilde{X}_{(k)}^{*}\|\leq
\frac{b_{0}^2a_{0}}{v}\|n^{-1}\widetilde{X}\widetilde{X}^{*}\|
\end{eqnarray*}
so that
\begin{equation*}
\mE \|Q_{n}(z)\|^2\leq
\frac{b_{0}^2a_{0}}{v}\mE\|n^{-1}\widetilde{X}\widetilde{X}^{*}\|\leq
\frac{b_{0}^2a_{0}}{v}(1+c\sqrt{p/n})\leq M.
\end{equation*}
Therefore, combining (\ref{eg:bound_on_d_3_plus_d_4}) and
(\ref{eg:bound_on_d_1}) we derive that $\mE |d_{k}^{(2)}|\leq M/p$. This,
together with (\ref{eq:E_tau_kk_bound}), implies that $d_{31}\leq M/p$.

\subsection{Bound on $d_{32}$}\label{d_32}
Denote by $\mE_{j}(\cdot)=\mE(\cdot|\mathcal{F}_{j})$ the conditional
expectation with respect to the $\sigma$-field generated by the first $j$ rows
of $X_{n}=(\mbf{x}_1^*, \cdots, \mbf{x}_{p}^*)^{*}$ except for
$\mbf{x}_{k}^{*}$, say, $\mathcal{F}_{j}=\sigma(\{\mbf{x}_l:1\leq l \leq j,
l\neq k\})$. Let $X_{(k,j)}=X_{(k)}-e_{j}\mbf{x}_{j}^{*}$, where $e_j$ denotes
the vector in $\mathbb{R}^p$ with 1 in $j$-th coordinate and zero elsewhere.
Then,
\begin{equation*}
\tr\left[Y^{-1}_{(k)}(z)\Lambda_{(k)}\right]-\mE
\tr\left[Y^{-1}_{(k)}(z)\Lambda_{(k)}\right]=\sum_{j\ne
k}^{p}\left[\mE_{j}\tr\left(Y_{(k)}^{-1}(z)\Lambda_{(k)}\right)-\mE_{j-1}\tr\left(Y_{(k)}^{-1}(z)\Lambda_{(k)}\right)\right]:=\sum_{
j\ne k}^{p}\gamma_{j}
\end{equation*}
where $\{\gamma_{j}\}$ forms a martingale difference sequence and can be
written as
\begin{eqnarray}\label{eq:martingale_gamma_j_expression}
\gamma_{j}&=&\mE_{j}\tr\left(Y_{(k)}^{-1}(z)\Lambda_{(k)}\right)-\mE_{j-1}\tr\left(Y_{(k)}^{-1}(z)\Lambda_{(k)}\right)\nonumber\\
&=& \mE_{j}\tr\left(Y_{(k)}^{-1}(z)\Lambda_{(k)}\right)-\mE_{j}\tr\left(Y_{(k,j)}^{-1}(z)\Lambda_{(k)}\right)\nonumber\\
&+& \mE_{j-1}\tr\left(Y_{(k,j)}^{-1}(z)\Lambda_{(k)}\right)-\mE_{j-1}\tr\left(Y_{(k)}^{-1}(z)\Lambda_{(k)}\right)\nonumber\\
&=&(\mE_{j}-\mE_{j-1})\left[\tr\left(Y_{(k)}^{-1}(z)\Lambda_{(k)}\right)-\tr\left(Y_{(k,j)}^{-1}(z)\Lambda_{(k)}\right)\right].
\end{eqnarray}
The second equality above holds because of the fact that
$$\mE_{j}\tr\left(Y_{(k,j)}^{-1}(z)\Lambda_{(k)}\right)=\mE_{j-1}\tr\left(Y_{(k,j)}^{-1}(z)\Lambda_{(k)}\right).$$
Thus, by Lemma \ref{lemma:quad_bound} we get
$$
\left|\tr\left(Y_{(k)}^{-1}(z)\Lambda_{(k)}\right)-\tr\left(Y_{(k,j)}^{-1}(z)\Lambda_{(k)}\right)\right|\leq
\frac{3a_{0}}{v}
$$
and hence $|\gamma_{j}|\leq 6a_{0}/v$ by
(\ref{eq:martingale_gamma_j_expression}). Applying Burkh\"{o}lder inequality
(Lemma \ref{eg:burkholder_inequality}), we have
\begin{eqnarray*}
d_{32}&=&\mE\left|\bar{b}_{2}(n)\frac{\lambda_{k}}{p}\tr\left[Y^{-1}_{(k)}(z)\Lambda_{(k)}\right]
- \bar{b}_{2}(n)\frac{\lambda_{k}}{p}\mE \tr\left[Y^{-1}_{(k)}(z)\Lambda_{(k)}\right]\right|^2\\
&=&\frac{(\bar{b}_{2}(n)\lambda_{k})^2}{p^2}\mE\left|\sum_{j\ne k}\gamma_{j}\right|^2\\
&\leq& \frac{K}{p^2}\mE\left(\sum_{j\ne k} |\gamma_{j}|^2\right)^{1/2}\\
&\leq& \frac{K}{p^2v^2}\left[\frac{36(p-1)a_{p}^2}{v^2}\right]^{1/2}\leq
\frac{M}{p^{3/2}}.
\end{eqnarray*}

\subsection{Calculation on extension to non-Gaussian case}
Since
\begin{eqnarray*}
\frac{\p C_{n}}{\p X_{ij}}&=&\frac{1}{\sqrt{np}}\lim_{\epsilon\to 0}
\frac{\left(A_{p}^{1/2}(X_n+\epsilon \bigtriangleup_{ij})B_n(X_n+\epsilon\bigtriangleup_{ij})^{*}A_{p}^{1/2}-tr(B_n)A_{p}\right)-\left(A_{p}^{1/2}X_nB_nX_n^{*}A_{p}^{1/2}-tr(B_n)A_{p}\right)}{\epsilon} \\
&=&\frac{1}{\sqrt{np}}\left(A_{p}^{1/2}X_{n}B_{n}\bigtriangleup_{ji}A_{p}^{1/2}+A_{p}^{1/2}\bigtriangleup_{ij}B_{n}X_{n}^{*}A_{p}^{1/2}\right)\\
&=&
\frac{1}{\sqrt{np}}\left(A_{p}^{1/2}X_{n}B_{n}\widetilde{e}_{j}e_{i}^{T}A_{p}^{1/2}
+ A_{p}^{1/2}e_{i}\widetilde{e}_{j}^{T}B_{n}X_{n}^{*}A_{p}^{1/2}\right),
\end{eqnarray*}
in which $e_{i}$ is a $p\times 1$ unit vector with 1 in $i$-th coordinate and
$\widetilde{e}_{i}$ is a $n\times 1$ unit vector with 1 in $i$-th coordinate.
Let $r_j = A_{p}^{1/2} X_{n} B_{n} \widetilde{e}_{j}$ and $\xi_{i} =
A_{p}^{1/2} e_{i}$. Then
\begin{equation*}
\frac{\p C_{n}}{\p X_{ij}} =
\frac{1}{\sqrt{np}}\left(r_{j}e_{i}^{T}A_{p}^{1/2}+
A_{p}^{1/2}e_{i}r_{j}^{*}\right) =
\frac{1}{\sqrt{np}}\left(r_{j}\xi_{i}^{*}+\xi_{i} r_{j}^{*}\right),
\end{equation*}
\begin{equation*}
\frac{\p^{2}C_{n}}{\p
X_{ij}^{2}}=\frac{2}{\sqrt{np}}b_{jj}A_{p}^{1/2}e_{i}e_{i}^{T}A_{p}^{1/2}=\frac{2}{\sqrt{np}}b_{jj}\xi_{i}\xi_{i}^{*},
\qquad \frac{\p^{3}C_{n}}{\p X_{ij}^{3}}=0,
\end{equation*}
\begin{equation*}
\frac{\p^2 G_{n}}{\p X_{ij}^{2}} = -\frac{\p^2 C_{n}}{\p
X_{ij}^2}G^{2}_{n}+2\frac{\p C_{n}}{\p X_{ij}}G_{n}\frac{\p C_{n}}{\p
X_{ij}}G^{2}_{n}= \frac{2}{\sqrt{np}}b_{jj}\xi_{i}\xi_{i}^{*}G_n^2,
\end{equation*}
\begin{eqnarray*}
\frac{\p^3 G_{n}}{\p X_{ij}^{3}} &=& -\frac{\p^3 C_{n}}{\p X^3_{ij}}G_{n}^{2}(z)+2\frac{\p^2 C_{n}}{\p X_{ij}^2}G_{n}\frac{\p C_{n}}{\p X_{ij}}G_{n}^2+2\frac{\p^2 C_{n}}{\p X^2_{ij}}G_{n}\frac{\p C_{n}}{\p X_{ij}}G^{2}_{n}\\
&-& 2\frac{\p C_{n}}{\p X_{ij}}\frac{\p C_{n}}{\p X_{ij}}G^{2}_{n}\frac{\p C_{n}}{\p X_{ij}}G^{2}_{n}+2\frac{\p C_{n}}{\p X_{ij}}G_{n}\frac{\p^2 C_{n}}{\p X_{ij}^2}G^{2}_{n}- 4 \frac{\p C_{n}}{\p X_{ij}}G_{n}\frac{\p C_{n}}{\p X_{ij}}G_{n}\frac{\p C_{n}}{\p X_{ij}}G^{2}_{n}\\
&=& 4\frac{\p^2 C_{n}}{\p X^2_{ij}}G_{n}\frac{\p C_{n}}{\p X_{ij}}G^{2}_{n}-2\frac{\p C_{n}}{\p X_{ij}}\frac{\p C_{n}}{\p X_{ij}}G^{2}_{n}\frac{\p C_{n}}{\p X_{ij}}G^{2}_{n}\\
&& + 2\frac{\p C_{n}}{\p X_{ij}}G_{n}\frac{\p^2 C_{n}}{\p X_{ij}^2}G^{2}_{n}- 4
\frac{\p C_{n}}{\p X_{ij}}G_{n}\frac{\p C_{n}}{\p X_{ij}}G_{n}\frac{\p
C_{n}}{\p X_{ij}}G^{2}_{n}.
\end{eqnarray*}
So we get
\begin{equation*}
\frac{1}{p}\tr\left[\frac{\p^3 G_{n}}{\p X_{ij}^{3}} \right]=
\frac{6}{p}\tr\left[\frac{\p C_{n}}{\p X_{ij}}G_{n}\frac{\p^2 C_{n}}{\p
X_{ij}^{2}}G_{n}^{2}\right]-\frac{6}{p}\tr\left[\frac{\p C_{n}}{\p
X_{ij}}G_{n}\frac{\p C_{n}}{\p X_{ij}}G_{n}\frac{\p C_{n}}{\p
X_{ij}}G_{n}^2\right]
\end{equation*}
where
\begin{equation*}
\frac{1}{p}\tr\left[\frac{\p C_{n}}{\p X_{ij}}G_{n}\frac{\p^2 C_{n}}{\p
X_{ij}^{2}}G_{n}^{2}\right]
=\frac{12b_{jj}}{np^2}\tr\left[\xi_{i}^{*}G^{2}_n(z)(r_{j}\xi_{i}^{*}+\xi_{i}r_{j}^{*})G_{n}\xi_{i}\right]
\end{equation*}
and
\begin{eqnarray*}
\frac{1}{p}\tr\left[\frac{\p C_{n}}{\p X_{ij}}G_{n}\frac{\p C_{n}}
{\p X_{ij}}G_{n}\frac{\p C_{n}}{\p X_{ij}}G^{2}_{n}\right]
&=&
\frac{1}{n^{3/2}p^{5/2}} \tr\left[(r_{j}\xi_{i}^{*} + \xi_{i}r_{j}^{*})G_{n}(r_{j}\xi_{i}^{*}
+ \xi_{i}r_{j}^{*})G_{n}(r_{j}\xi_{i}^{*}+\xi_{i}r_{j}^{*})G_{n}^{2}\right]\\
&:=& 2\eta_{3}(n)+2\eta_{4}(n)+2\eta_{5}(n)+2\eta_{6}(n)
\end{eqnarray*}
where
\begin{eqnarray*}
\eta_{3}(n)&=& \frac{1}{n^{3/2}p^{5/2}}\left[(r_{j}^{*}G_{n}\xi_{i})^2r_{j}^{*}G_{n}^{2}\xi_{i}\right]\\
\eta_{4}(n) &=& \frac{1}{n^{3/2}p^{5/2}}\left[r_{j}^{*}G_{n}\xi_{i}r_{j}^{*}G_{n}r_{j}\xi_{i}^{*}G_{n}^{2}\xi_{i}\right]\\
\eta_{5}(n) &=&  \frac{1}{n^{3/2}p^{5/2}}\left[r_{j}^{*}G_{n}r_{j}\xi_{i}^{*}G_{n}\xi_{i}r_{j}^{*}G^{2}_{n}\xi_{i}\right]\\
\eta_{6}(n)&=&
\frac{1}{n^{3/2}p^{5/2}}\left[\xi_{i}^{*}G_{n}\xi_{i}r_{j}^{*}G_{n}\xi_{i}r_{j}^{*}G^{2}_{n}r_{j}\right].
\end{eqnarray*}

\subsection{Proof of Lemma \ref{lem:estimation_r}}

Let $B_{\cdot j}=B_{n}\widetilde{e}_{j}=(b_{1}, \cdots, b_{n})^{T}$ (for
brevity, dropping index $j$ on the right) and $M_{n}=A_{p}^{1/2}X_{n}$. Since
$r_{j}= M_{n}B_{\cdot j}$, where $\mE M_{ij}=0$ and
$\mE|M_{ij}|^{2}=(A_{p})_{ii}$, we have
\begin{equation*}
\mE\left(\|r_{j}\|^2\right)^{k}=\mE \left(B_{\cdot j}^{*}M_{n}^{*}M_{n}B_{\cdot
j}\right)^{k} = \mE
\left(\sum_{i=1}^{p}\left[\sum_{k=1}^{n}M_{ik}b_{k}\right]^2\right)^{k} = \mE
\left(\sum_{i=1}^{p}N_{i}^{2}\right)^{k},
\end{equation*}
where $N_{i}$, $i=1, \cdots p$, are independent, sub-Gaussian random variables
with $\mE N_{i}=0$ and $\mE N_{i}^{2}=(A_{p})_{ii}\|B_{\cdot j}\|^2$. Then we
have
\begin{equation*}
\mE
\left(\sum_{i=1}^{p}N_{i}^2\right)^{k}=\mE\left(\sum_{i=1}^{p}(N_{i}^{2}-\mE
N_{i}^2)+ \|B_{\cdot j}\|^2 \tr(A_{p})\right)^{k},
\end{equation*}
where $N^2_{i}-\mE N^2_{i}$ is a mean zero sub-exponential random variable.
Thus,
\begin{equation*}
\frac{1}{p^k}\mE
\left(\sum_{i=1}^{p}N_{i}^2\right)^{k}=\mE\left(\frac{1}{p}\sum_{i=1}^{p}(N_{i}^{2}-\mE
N_{i}^2)+\|B_{\cdot j}\|^2 \frac{\tr(A_{p})}{p}\right)^{k}=O(1).
\end{equation*}
The term $\|B_{\cdot j}\|^2 \tr(A_{p})/p\leq a_{0}b^{2}_{0} = O(1)$. On the
other hand, $\frac{1}{p}\sum_{i=1}^{p}(N_{i}^{2}-\mE N^{2}_{i})$ is the average
independent sub-exponential random variables with mean zero and uniformly
bounded sub-exponential norm (can be verified). So by Bernstein's inequality
(Lemma \ref{lem:vershynin_berstein_tail}), the tail probability can be
controlled adequately so that
$\mathbb{E}(\frac{1}{p}\sum_{i=1}^{p}(N_{i}^{2}-\mE N^{2}_{i}))^k = O(1)$ for
any $k \geq 1$. Hence (\ref{r}) holds.


\end{document}